\documentclass[12pt]{amsart}

\usepackage{amsmath,amsfonts,amssymb,amscd,amsthm,latexsym}

\usepackage{pstricks,pst-node,pst-plot}
\usepackage{graphicx}

\newcommand{\dee}{\mbox{$\mathcal{D}$}}

\newcommand{\jay}{\mbox{$\mathcal J$}}
\newcommand{\ar}{\mbox{$\mathcal R$}}
\newcommand{\el}{\mbox{$\mathcal L$}}
\newcommand{\eh}{\mbox{$\mathcal H$}}

\newcommand{\els}{\mbox{${\mathcal L}^{\ast}$}}

\newcommand{\ars}{\mbox{${\mathcal R}^{\ast}$}}

\newcommand{\s}{\mbox{$\mathcal{S}$}}

\newtheorem{theorem}{Theorem}[section]

\newtheorem{corollary}[theorem]{Corollary}

\newtheorem{lemma}[theorem]{Lemma}

\newtheorem{proposition}[theorem]{Proposition}

\newcommand{\tee}{\mathcal{T}}

\begin{document}

\title[Model-theoretic properties of free, projective and flat $S$-acts
 ]{Model-theoretic properties of free, projective and flat $S$-acts}
\subjclass{20 M 30, 03 C 60}
%\thanks{{\em to add}}
\date{\today}

\author{Victoria Gould, Alexander Mikhalev, Evgeny Palyutin, Alena Stepanova}
{\footnotetext[1]{This research was supported by RFBR (grant
17-01-00531)}
\address{Department of Mathematics\\University
 of York\\Heslington\\York YO1 5DD\\UK}
\email{varg1@york.ac.uk}
\address{Faculty of Mechanics and Mathematics\\Lomonosov Moscow State University\\Moscow\\Russia}
\email{mikhalev@rector.msu.ru}
\address{Department of Algebra\\Sobolev Institute of Mathematics\\Novosibirsk\\Russia}
\email{palyutin@math.nsc.ru}
\address{Institute of Mathematics and Computer Science\\Far East State
University\\Vladivostok\\Russia}
\email{stepltd@mail.primorye.ru}

\begin{abstract} This is the second in a series of articles
surveying the body of work on the model theory of $S$-acts over a
monoid $S$. The first concentrated on the theory of {\em regular}
$S$-acts. Here we review the material  on model-theoretic properties
of {\em free, projective} and {\em (strongly, weakly) flat}
$S$-acts. We consider  questions of axiomatisability, completeness,
model completeness and stability for these classes. Most but not all
of the results have already appeared; we remark that the description
of those monoids $S$ such that the class of free left $S$-acts is
axiomatisable, is new.
 \end{abstract}
\maketitle

%\noindent L-¦~510.67:512.56
\sloppy  \vskip 1cm

\section{Introduction}

The interplay between model theory and other branches of
mathematics is a fruitful and fascinating area. The model theory
of modules over a ring $R$ has long been a respectable branch of
both model theory and ring theory: an excellent introduction to
the subject may be found in the book of Prest \cite{mprest}. The
model theory of acts over monoids is rather less developed but
again exhibits a nice interplay between  algebra and  model
theory, with its own distinct flavour. In an attempt to make the
existing results available to wider audiences, a group of authors
is engaged in writing a series of survey articles, of which this
is the second. The first \cite{MOPS} concentrated on the theory of
{\em regular} $S$-acts.

A left $S$--act over monoid $S$ is a set $A$ upon which $S$ acts
unitarily  on the
left.  Thus,
left $S$--acts can be considered as a  natural generalisation of  left
modules over rings. Certainly many questions that can be asked
and answered in the
model theory of modules can be asked for $S$-acts,
but often have rather different answers.
For example, there is a finite ring $R$ such that the class of
free left $R$-modules is not axiomatisable, whereas
if $S$ is a finite monoid then the class of free left
$S$-acts is {\em always} axiomatisable.
 At a basic level, the major difference is that there
is no underlying group structure to an $S$-act, so that
congruences cannot in general be determined by special subsets.

A class of $L$-structures $\mathcal{C}$ for a first order language
$L$ is {\em axiomatisable} or {\em elementary} if there is a set of
sentences $\Pi$ in $L$ such that an $L$-structure $\mathbf{A}$ lies
in $\mathcal{C}$ if and only if every sentence in $\Pi$ is true in
$\mathbf{A}$. Eklof and Sabbah \cite{es} characterise those rings
$R$ such that the class of all flat (projective, free) left
$R$-modules is
 axiomatisable (in the natural first order language
associated with $R$-modules). What is at scrutiny here is the power of
a first order language to characterise categorical notions. Naturally
enough the conditions that arise are finitary conditions on the ring
$R$. In the theory of
$S$--acts there are three
contenders to the notion of flatness, called here
weakly flat, flat and strongly flat,
for which corresponding analogues for modules over a ring
 coincide. We devote Sections 5 to 9 to characterising those
monoids $S$ such that the classes of weakly flat, flat, strongly flat,
projective and free $S$-acts are axiomatisable. The material for
Sections~\ref{wfsection} to ~\ref{projsection}
 is taken from the papers \cite{gould1} and \cite{BG}
of the first author and Bulman-Fleming, and the paper \cite{stepanova}
of the fourth author. Section~\ref{freesection}
contains a new result  of the first author characterising those monoids
such that the class of free left $S$-acts is axiomatisable,
and some
specialisations taken from \cite{stepanova}.

A theory $T$ in $L$ is {\em complete} if for each sentence $\varphi$
of $L$ we have either $\varphi\in T$ or $\neg\varphi\in T$. This is
equivalent to saying that for any models $\mathbf A$ and $
\mathbf B$ of $T$, that
is, $L$-structures $\mathbf A$ and $\mathbf B$ in which all sentences of $T$ are
true, we have that $\mathbf A$ and $\mathbf B$ are `the same' in some sense; precisely,
they are {\em elementarily equivalent}. Related notions are those
of {\em model completeness} and {\em categoricity}. We can define
associated notions of completeness, model completeness and categoricity
for
classes of $L$-structures. Section~\ref{comp},
the results of which are
taken from \cite{stepanova}, considers
 the question of when the classes of strongly flat, projective and
free left $S$-acts are
complete, model complete or categorical, given that they are
axiomatisable.

Sections~\ref{stability} and ~\ref{morestability}
investigate the crucial properties of
{\em stability} (surveyed
in Section~\ref{stabintro}) for our classes of $S$-acts.
We remarked above that although a parallel
can be drawn between some problems in
the model theory of modules and that for $S$--acts,
the theories soon diverge. This is immediately apparent
when stability comes into question.
Each complete theory of an
$R$-module over
a ring $R$  is stable, whereas
there are $S$-acts with unstable theories \cite{mus,fg,ste2}.
We consider questions of stability for
classes of strongly flat, projective and free left $S$-acts. Let us
say that a class
$\mathcal{C}$ of structures is stable (superstable,
$\omega$--stable) if for any $\mathbf{A}\in \mathcal{C}$
the set of sentences true in $\mathbf{A}$ (which is a complete
theory)  is stable (superstable,
$\omega$--stable). It is proved that if the class of strongly flat left
$S$-acts is axiomatisable and $S$ satisfies an additional finitary
condition known as Condition (A), then the class is superstable. Moreover, if the class of
projective (free) left $S$-acts is axiomatisable, then it is
superstable.
Finally we consider the question of $\omega$-stability for
an axiomatisable class of
 strongly flat left $S$-acts, given that $S$ is countable and satisfies
Condition (A). Consequently, if the
class of projective (free) left $S$--acts is axiomatisable for a
countable monoid $S$, then it is
$\omega$--stable.

We have endeavoured to make this paper as self-contained
 as possible, giving
all necessary definitions and results,
with  proofs or careful references
to
background results where they are not immediately
available. The article should
be accessible to readers with only
 a  rudimentary knowledge of semigroup theory,
acts over monoids and
model theory - little more than the notions of ideal, subsemigroup,
$S$-act,
first-order language and  interpretation. We devote Section~\ref{ma}
to an introduction to the necessary general
background for monoids and acts and Section~\ref{fpf} to a specific
discussion of how the classes of flat, projective and free $S$-acts
arise, their properties etc. Section~\ref{ax} gives further details
concerning axiomatisable classes, including a discussion of
ultraproducts
of specific kinds. Section~\ref{cmc} gives the necessary background for
complete, model complete and categorical theories, and as mentioned
above,
Section~\ref{stabintro} contains a discussion of the notion of
stability.
For a more comprehensive treatment
we recommend the reader to
 \cite{howie} (for semigroup theory),
\cite{kilp} (for the theory of $S$-acts) and \cite{changkeisler} (for
model theory).

The finitary conditions that arise
from considerations of axiomatisability etc.,
and the monoids that satisfy them,  are of interest
in themselves. They are currently the subject of the PhD of
L. Shaheen,
a student of the first author. The only serious
omission of material in this article
is that we have preferred not to recall the
examples
presented in the papers from which much of this
article is constituted, and refer the reader to the references
for further information.

\section{Monoids and acts}\label{ma}

 Throughout this paper $S$ will denote a monoid
with identity $1$ and set of idempotents $E$. Maps will be written on
the
{\em left} of their arguments.
We make frequent use
of  the five equivalences on $S$
known as {\em Green's relations}
$\ar,\el,\eh,\dee$ and $\jay$. For the convenience we recall here that
the relation $\ar$ is defined on $S$ by the rule that for any
$a,b\in S$,
\[a\,\ar\, b\mbox{ if and only if }
aS=bS.\]
Clearly, $a\,\ar\, b$ if and only if $a$ and $b$ are mutual left
divisors,  and $\ar$ is a left compatible equivalence relation. The
relation
$\el$ is defined dually.  The meet
$\eh$ of $\ar$ and $\el$ (in the lattice of equivalences on $S$)
is given by  $\eh=\ar\cap\el$. Immediately from Proposition 2.1.3 of
\cite{howie}, the join $\dee$ of $\ar$ and $\el$ is given by
\[\dee=\ar\circ \el=\el\circ \ar.\]
The fifth relation, $\jay$, is defined by the rule that for any
$a,b\in S$, $a\jay b$ if and only if $SaS=SbS$.

 The following result is due to Green.

\begin{theorem}\label{green}
(Theorem 2.2.5) \cite{howie} If $H$ is an $\eh$-class of
 $S$, then either $H^2\cap H=\emptyset$ or $H^2=H$ and $H$ is a
subgroup of $S$.
\end{theorem}

It is standard convention to write $K_a$ for the $\mathcal{K}$-class
of an element $a\in S$, where $\mathcal{K}$ is one of Green's relations.
Theorem~\ref{green} gives immediately that $H_e$ is a subgroup, for any
$e\in E$. The {\em group of units} of $S$ is $H_1$; we say that
$S$ is {\em local} if $S\setminus H_1$ is an ideal. The following lemma
is straightforward.

\begin{lemma}\label{localideal} The monoid $S$ is local if and only if
\[R_1=L_1=H_1.\]
\end{lemma}

Clearly
$\dee\subseteq \jay$; it is not true  for a general monoid that $\dee=\jay$. However, the
equality holds for finite monoids and, more generally, for epigroups.
We say
a monoid $S$ is an {\em epigroup} (or {\em  group bound}) if every
element of $S$ has a positive power that lies in a subgroup. The following result
is well known but we include for completeness its proof.

\begin{proposition}\label{groupboundlocal}
Let $S$ be a group bound monoid.
Then $S$ is local
and $\mathcal{D}=\mathcal{J}$.
\end{proposition}
\begin{proof} Let $a,b\in S$ with $a\,\jay\, b$. Then there exist
$x,y,u,v\in S$ with $a=xby, b=uav$. We have that
\[a=xby=x(uav)y=(xu)a(vy)=(xu)^na(vy)^n\]
for any $n\in\omega$. By
hypothesis we may pick $n\geq 1$ with $(xu)^n$ lying in a subgroup.
By Theorem~\ref{green} we have that $(xu)^n\,\el\, (xu)^{2n}$,
whence as $S(xu)^{2n}\subseteq S(xu)^{n+1}\subseteq S(xu)^n$ we must
have that $(xu)^{n+1}\,\el\, (xu)^n$. Since $\el$ is a left congruence
we have that
\[a=(xu)^na(vy)^n\,\el\, (xu)^{n+1}a(vy)^n=xua.\]
By the same argument as above, $a\,\el\, ua$; dually, $a\,\ar\, av$.
Hence
\[a\,\el\, ua\, \ar\, uav=b\]
so that $a\,\dee\, b$ and $\dee=\jay$ as required.

A similar argument yields that  $S$ has the `rectangular' property, that is, if
$a\,\dee\, b\,\dee\, ab$, then $a\,\ar\, ab\,\el\, b$.

If $1\,\dee\, a$, then as certainly $1\, a=a$, we have that
$1\,\ar\, a$; dually, $1\,\el\, a$. Hence $1\,\eh\, a$ and
\[J_1=D_1=L_1=R_1=H_1.\]
\end{proof}

A monoid $S$ is {\em regular} if for any $a\in S$ there exists an
$x\in S$ with $a=axa$. Notice that in this case, $ax,xa\in E$ and
\[ax\,\ar\, a\,\el\, xa.\]
A regular monoid $S$ is {\em inverse} if, in addition, $ef=fe$ for all
$e,f\in E$. In an inverse monoid the idempotents form a commutative
subsemigroup which we refer to as the {\em semilattice of idempotents}
of $S$.

 A monoid
$S$ is {\em right collapsible} if for any $s,t\in S$ there
exists $u\in S$ such that $su=tu$.

\begin{lemma}\label{rightcollapsible} Let $S$ be a right collapsible
monoid and let $a_1,\hdots ,a_n\in S$. Then there exists
$r\in S$ with
\[a_1r=a_2r=\hdots =a_nr.\]
\end{lemma}
\begin{proof} Certainly the result is true for $n=1$ or $n=2$. If
$a_1r=a_2r=\hdots =a_ir$ for some $i$ with $2\leq i<n$, then pick
$s\in S$ with $a_irs=a_{i+1}rs$ and note that $a_1rs=\hdots =a_irs=
a_{i+1}rs$. The result follows by finite induction.

\end{proof}

 A submonoid $T$ of $S$ is {\em
right unitary} if for any $t\in T,s\in S$, if $st\in T$, then $s\in
T$. We say that a monoid has the {\em condition of finite
right solutions}, abbreviated by {\em (CFRS)}, if
\[\forall s\in S\; \exists n_s\in\mathbb{N} \; \forall t\in S
\;|\{ x\in S|\; sx=t\}|\leq n_s.\]

 The {\em
starred} analogues $\els$ and $\ars$ of $\el$ and $\ar$ also figure
significantly in this work. We recall that elements $a,b$ of $S$ are
 $\ars$-related if for all $x,y\in S$,
\[xa=ya\mbox{ if and only if }xb=yb.\]
Clearly, $\ars$ is a left congruence on $S$ containing $\ar$;
it is not hard to see that $\ar=\ars$ if $S$ is regular. The relation
$\els$ is defined dually.

\bigskip

This article is concerned with model theoretic aspects
of the representation theory of
monoids via morphisms to monoids of self-mappings of sets.
We recall that
a set $A$ is  a {\em left $S$-act} if there is a map $S\times
A\rightarrow A, (s,a)\mapsto sa$, such that for all $a\in A$
and $s,t\in S$,
\[1a=a\mbox{ and }s(ta)=(st)a.\]
Terminology for $S$-acts has not been consistent in the literature: they
are known variously as {\em $S$-sets, $S$-systems, $S$-operands}
and {\em $S$-polygons}. The definitive reference \cite{kilp} uses the
term $S$-act, as do we here.

To say that a set $A$ is a left $S$-act is equivalent to there being
 a morphism from $S$ to the full
transformation
monoid $\mathcal{T}_A$ on $A$. For the record, we denote the identity map
on a set $X$ by $I_X$, so that such a morphism must take
$1\in S$ to $I_A\in\mathcal{T}_A$. An $S$-{\em morphism} from a left
$S$-act
$A$ to a left $S$-act $B$ is a map $\theta:A\rightarrow B$ such that
$\theta(sa)=s\theta(a)$, for all $a\in A,s\in S$. Clearly,
the class of left $S$-acts together with $S$-morphisms  forms
a category, $\mathbf{S}${\bf -Act}. It is clear that in
$\mathbf{S}${\bf -Act} the
coproduct
of $S$-acts $A_i,i\in I$ is simply disjoint union, denoted
$\coprod_{i\in I}A_i$. Right $S$-acts, and the category
{\bf Act-}${\mathbf S}$, are defined dually.

An $S$-{\em subact} of a left $S$-act $A$ is a subset $B$ of $A$
closed under the action of $S$. Clearly $S$ may be regarded as a
left $S$-act, and any left ideal becomes an $S$-subact. A left
$S$-act $A$ is {\em finitely generated} if there exists $a_1,\hdots
,a_n\in A$ such that $A=\bigcup_{i=1}^{n}Sa_i$ and {\em cyclic} if
$A=Sa$ for some $a\in A$. If $x\in X$, where $X$ is a set disjoint
from $S$, then the set of formal expressions $Sx=\{ sx\mid s\in S\}$
becomes a cyclic left $S$-act in an obvious way. Notice that
$Sx\cong S$. The proof of the following lemma is immediate.

\begin{lemma}\label{iso} Let $A,B$ be a left $S$-acts, let $a\in A$ and
$b\in B$.
Then there is an $S$-isomorphism $\theta:Sa\rightarrow Sb$ such that $
\theta(a)=b$ if and only if for all $x,y\in S$,
\[xa=ya\mbox{ if and only if }xb=yb.\]
\end{lemma}

Lemma~\ref{iso} gives in particular
 that if $e\in E$, then $Sa\cong Se$
under an isomorphism $\theta$ such that
$\theta(a)=e$,  if and only if
$ea=a$ and for all $x,y\in S$, $xa=ya$ implies that $xe=ye$; such an
$a$ is said to be {\em right $e$-cancellable}.

 From \cite{abundant}, or Lemma~\ref{iso}
above, we deduce:

\begin{lemma}\cite{abundant}\label{pp} Let $a,b\in S$. The left ideals $Sa$ and $Sb$ are
isomorphic as left $S$-acts, under an isomorphism $\theta$ such that
$\theta(a)=b$, if and only if $a\,\ars\, b$.

\end{lemma}

 A monoid $S$
is called {\em right cancellative} if every element from $S$ is
right 1-cancellable. The notions of a left $e$-cancellable element
 and a left cancellative monoid are defined
dually.

A {\em congruence} on a left
$S$-act $A$ is an equivalence relation $\rho$ on $A$ such that
$(a,a^\prime)\in\rho$ implies $(sa,sa^\prime)\in\rho$ for
$a,a^\prime\in A$, $s\in S$.
To prevent ambiguity, a congruence on $S$ {\em regarded
as a left $S$-act} will be referred to as a {\em left congruence}.
If $\rho$ is a congruence on $A$ we
shall also write $a\, \rho\,  a^\prime$ for $( a,a^\prime)\in\rho$
and $a/\rho$ for the $\rho$-class of $a\in A$.
 If $X\subseteq A\times A$ then by $\rho (X)$ we
denote the smallest congruence on $A$ containing $X$.

\begin{proposition}\label{cong}\cite{kilp} Let $A$ be an $S$--act, $X\subseteq
A\times A$ and $\rho=\rho (X)$. Then for any $a,b\in A$, one has
$a\,\rho\, b$ if and only if either $a=b$ or there exist
$p_1,\ldots,p_n,q_1,\ldots,q_n\in A, s_1,\ldots,s_n\in S$
such that $(p_i,q_i)\in X$ or $(q_i,p_i)\in X$ for any $i$, $1\leq
i\leq n$, and
\[a=s_1 p_1,\;s_1q_1=s_2p_2,\; \hdots \; ,
s_nq_n=b.\]\end{proposition}

Elements $x,y$ of a left $S$--act $A$ are  {\em connected}
(denoted by $x\sim y$) if there exist $n\in\omega$,
$a_0,\ldots,a_n\in A$, $s_1,\ldots,s_n\in S$ such that $x=a_0$,
$y=a_n$, and $a_i=s_ia_{i-1}$ or $a_{i-1}=s_ia_i$ for any $i,\;1\leq
i\leq n$. An $S$-subact $B$ of a left $S$--act $A$ is
a {\em connected}
 if we
have $x\sim y$ for any $x,y\in B$. It is easy to check that
$\sim$ is a congruence relation on a left $S$--act $A$, with classes
that are $S$-subacts, and maximal connected components of $A$.
Thus, $A$ is
a coproduct of connected $S$-subacts.

\section{Free, Projective and Flat Acts}\label{fpf}

For the convenience of the reader we discuss in this section
 the classes
of free, projectives and  flat left $S$-acts; as explained below, there
are several candidates for the notion of a flatness. What we give is a
skeleton survey; further details may be found in \cite{kilp}.

We remind the reader that a left $S$-act $F$ is {\em free} (on a subset
 $X$, in
$\mathbf{S}${\bf -Act}) if and only if for any left $S$-act $A$ and map
$\theta:X\rightarrow A$, there is a unique $S$-morphism
$\overline{\theta}:F\rightarrow
A$ such that $\overline{\theta}\iota=\theta$, where $\iota:X\rightarrow
F$ is the inclusion mapping.

\begin{theorem}\label{free}\cite{kilp} A left $S$-act $F$ is free
on $X$ if and only if   $
F\cong\coprod_{x\in X}Sx$. \end{theorem}
Notice from the above that a free left $S$-act is isomorphic to a
coproduct of copies of the left $S$-act $S$.
From the remark following Lemma~\ref{iso} we have our next
corollary.

\begin{corollary}\label{freecyclic} A cyclic left $S$-act $A$ is free if and only if
$A=Sa$ for some right $1$-cancellable $a\in A$.

\end{corollary}

 A left $S$-act $P$ is {\em projective}
if given any diagram of left $S$-acts and $S$-morphisms

\begin{center}\begin{pspicture}(0,0)(2,2)
\psset{nodesep=3pt} \rput(0,0){\rnode{M}{$M$}}
\rput(2,0){\rnode{N}{$N$}} \rput(2,2){\rnode{P}{$P$}}
\rput(1.1,0.4){\rnode{K}{$\phi$}} \rput(2.3,1){\rnode{L}{$\theta$}}
\ncline{->}{M}{N} \ncline{->}{P}{N}
\end{pspicture}
\end{center}
where $\phi:M\rightarrow N$ is onto, there exists an $S$-morphism
$\psi:P\rightarrow M$ such that the diagram

\begin{center}\begin{pspicture}(0,0)(2,2)
\psset{nodesep=3pt} \rput(0,0){\rnode{M}{$M$}}
\rput(2,0){\rnode{N}{$N$}} \rput(2,2){\rnode{P}{$P$}}
\rput(1.1,0.5){\rnode{K}{$\phi$}} \rput(2.3,1){\rnode{L}{$\theta$}}
\rput(0.9,1.3){\rnode{T}{$\psi$}} \ncline{->}{M}{N}
\ncline{->}{P}{N} \ncline{->}{P}{M}
\end{pspicture}
\end{center}
is commutative.

It is clear that a free left $S$-act $F$ is projective; this is also an
immediate
consequence of Theorem~\ref{free} and Theorem~\ref{proj} below.

\begin{theorem}\label{proj}\cite{knauer,dorofeeva} A left $S$-act $P$ is
projective if and only if $P\cong \coprod _{i\in I}Se_i$,
where $e_i\in E$ for all $i\in I\neq\emptyset$.
\end{theorem}

\begin{corollary}\label{cyclicproj} A cyclic $S$-set $A$ is projective if and only if
$A=Sa$ for some right $e$-cancellable $a\in S$, for some $e\in E$.

\end{corollary}

To define classes of flat left $S$-acts we need the notion of {\em
tensor product} of $S$-acts. If $A$ is a {\em right} $S$-act and $B$
a left $S$-act then the tensor product of $A$ and $B$, written
$A\otimes B$, is the set $A\times B$ factored by the equivalence
generated by $\{ ((as,b),(a,sb))\mid a\in A,b\in B,s\in S\}$. For
$a\in A$ and $b\in B$ we write $a\otimes b$ for the equivalence
class of $(a,b)$.

For a left $S$-act $B$,
the map $-\otimes B$ is a functor from the category
{\bf Act--}$\mathbf{S}$ to the category {\bf Set} of sets. It is from this
functor that  the various notion of {\em flatness} are derived.

A left $S$-act $B$ is {\em weakly flat} if the functor
$-\otimes B$ preserves embeddings of right ideals of
$S$ into $S$, {\em flat} if it preserves arbitrary
embeddings of right $S$-acts, and {\em strongly flat}
if it preserves  pullbacks (equivalently, equalisers
and pullbacks \cite{bf}).
We give the reader a warning that terminology has
changed over the years; in particular,
 left $S$-acts $B$ such that
$-\otimes B$ preserves equalisers and pullbacks
are called weakly flat in \cite{fountain} and \cite{stenstrom},
 flat in \cite{gould1},
and pullback flat in \cite{kilp}.

Stenstr\"{o}m was instrumental in forwarding the theory of flat acts,
by producing interpolation conditions, later labelled (P) and (E), that
are together equivalent to strong flatness.

\begin{theorem}\label{stronglyflat} \cite{stenstrom} A left $S$-act $B$ is
strongly flat if and only if $B$ satisfies conditions
(P) and (E):

  (P): if
$x,y\in B$ and $s,t\in S$ with
 $sx=ty$, then there is an element $z\in B$
and elements $s',t'\in S$ such that $x=s'z,y=t'z$ and $ss'=tt'$;

 (E): if $x\in B$ and $s,t\in S$ with
$sx=tx$, then there is an element $z\in B$ and $s'\in S$ with
$x=s'z$ and $ss'=ts'$.
\end{theorem}

The classes of free, projective, strongly flat, flat and weakly flat
left $S$-acts will be denoted by $\mathcal{F}r, \mathcal{P,SF,F}$ and
$\mathcal{WF}$ respectively.
We remark here that
 `elementary' descriptions (that is, not involving arrows)
of $\mathcal{F}$ and $\mathcal{WF}$,
along the lines of Theorems~\ref{free},~\ref{proj} and
\ref{stronglyflat} for
$\mathcal{F}r,\mathcal{P}$ and $\mathcal{SF}$, are not available.
From those results it is immediate that
projective left $S$-acts are strongly flat.
Clearly flat left $S$-acts are
weakly flat and since embeddings
are equalisers in {\bf Act--}$\mathbf{S}$, strongly
flat left $S$-acts are flat. Thus
\[\mathcal{F}r\subseteq \mathcal{P}\subseteq\mathcal{SF}\subseteq\mathcal{F}
\subseteq\mathcal{WF}.\]

A congruence $\theta$ on a left $S$-act $A$ is called {\em strongly
flat} if $A/\theta\in\mathcal{SF}$. For the purposes of later
sections we state a  straightforward consequence of
Theorem~\ref{stronglyflat}.

\begin{corollary}\label{sfcong} \cite{bulmanflemingnormak} A congruence $\theta$ on the left $S$-act
$S$ is strongly flat if and only if for any $u,v\in S$, $u\,\theta\, v$
if and only if there exists $s\in S$ such that $s\,\theta\, 1$ and $us=vs$.

\end{corollary}

\section{Axiomatisability}\label{ax}

Any class of algebras $\mathcal{A}$ has an associated first order
language $L$. One can then ask whether a property $P$, defined for
members of $\mathcal{A}$, is expressible in the language  $L$. In
other words, is there a set of sentences $\Pi$ in $L$ such that $A$
is a member of $\mathcal{A}$ if and only if all the sentences of
$\Pi$ are true in $A$, that is, if and only if
  $A$ is  a model of
$\Pi$, which we denote by $A\vDash \Pi$.
 If $\Pi$ exists we say that $\mathcal{A}$ is
{\em axiomatisable}. Questions of axiomatisability are the first
step in investigating the {\em model theory} of a class of algebras.

We
 concentrate in this article on aspects of the first-order theory of
left
$S$-acts, the next  five sections considering questions of
axiomatisability of the classes of free, projective and (strongly,
weakly)
flat acts.
 Our language is the
first order language with equality $L_S$
which has no constant or relation symbols and which
has a unary function symbol $\lambda_s$ for each $s\in S$:
we write $sx$ for $\lambda_s(x)$.
The class of left $S$-acts is axiomatised by the set of
sentences $\Pi$ where
\[\Pi=\{ (\forall x)(1x=x)\} \cup\{ \mu_{s,t}\mid s,t\in S\}\]
where $\mu_{s,t}$ is the sentence
\[(\forall x)((st)x=s(tx)).\]

Certain classes of left $S$-acts
are axiomatisable for {\em any} monoid $S$. For
example the torsion free left $S$-acts
are axiomatised by $\Pi\cup\Sigma_{\mathcal{TF}r}$
where
\[\Sigma_{\mathcal{TF}r}=
\{ (\forall x)(\forall y)(sx=sy\rightarrow x=y)\mid s\in T\}\] where
$T$ is the set of left cancellable elements of $S$. Indeed in the
context of $S$-acts the sentences of $\Pi$ are understood, and we
say more succinctly that $\Sigma_{\mathcal{TF}r}$ axiomatises
$\mathcal{TF}r$.
 Other
natural classes of left $S$-acts are axiomatisable for some
monoids and not for others.

One of our main tools throughout will be that of an  ultraproduct, which
we now briefly recall.

For any set $A$ we denote the set of all subsets of $A$ by
$P(A)$.  A family $C\subseteq P(A)$
is  {\em centred}  if for any $X_1,\dots,X_n\in C$ the
intersection $X_1\cap\dots\cap X_n$ is not empty. We remark that if
$C$ is centred, then $\emptyset\notin C$. A non-empty
family $F\subseteq P(A)$ is called a {\em filter} on $A$ if the
follow conditions are true:

(a) $\emptyset\notin F$;

(b) if $U,V\in F$ then $U\cap V\in F$;

(c) if $U\in F$ and $U\subseteq X\subseteq A$, then $X\in F$.

A filter $F$ on a set $A$ is said to be {\em uniform}  if $|X|=|A|$
for any $X\in F$ and an {\em
ultrafilter}  if $X\in F$ or $A\setminus X\in F$ for any
$X\subseteq A$. The following facts concerning these concepts
follow easily from the definitions.

(1) A filter $F$ on a set $A$ is a centred set and $A\in F$.

(2) If $F$ is an ultrafilter on $A$ and $U_0\cup\cdots\cup U_n\in
F$ then $U_i\in F$ for some $i\leqslant n$.

(3) If  $A$ is an infinite set then the set
$$\Phi_A=\{X\mid X\subseteq A,\;|A\setminus X|<|A|\}$$
is a filter (it is called the Fr\'{e}chet filter).

(4) An ultrafilter $F$ on an infinite set $A$ is uniform
if and only if $F$
contains the filter $\Phi_A$.

We now argue, that (with the use of Zorn's Lemma), every centred family
can be extended to an ultrafilter. Hence, there are ultrafilters on
every non-empty set.

\begin{proposition}\label{filter1}\label{centfam} Let $A$ be a non-empty set. Then:

(1) a centred family $C\subseteq P(A)$ extends to a maximal
centred family $D$;

(2) a  centred family is an ultrafilter if and
only if it is a maximal centred family;

(3) any centred family $C\subseteq P(A)$ is contained in some
ultrafilter $F$ on $A$;

(4) a filter
 is an ultrafilter if and only if it is a maximal filter;

(5) if $A$ is infinite then there is a uniform ultrafilter on $A$.
\end{proposition}
\begin{proof} (1)  Let $S$ be the
 set of all centered families of
subsets of
$A$ containing the family $C$. It is clear that the union of an
ascending chain of centered families is a centered family. Thus
the poset $\langle S,\subseteq\rangle$ satisfies the maximal condition
required to invoke
Zorn's Lemma. We deduce that  there exists a
centered family $D$ which is maximal in $\langle
S,\subseteq\rangle$.

(2) Let $C$ be a maximal centred family.
 We have observed that $\emptyset\notin C$, and
clearly $C$ is closed under intersection. If $X\in C$ and $X\subseteq Y
\subseteq A$, then clearly $C\cup \{ Y\}$ is centred. By maximality of
$C$ we deduce that $Y\in C$ and so $C$ is a filter.
 If $X\notin C$, then $C\cup \{ X\}$ is not centred, so that
$X\cap Y=\emptyset$ for some $Y\in C$. Hence $Y\subseteq A\setminus X$
so that $A\setminus X\in C$ since $C$ is a filter. We deduce that $C$
is an ultrafilter.

Conversely, any ultrafilter $F$ is centred, and if $F\subset G$ for a
centred
family $G$, then taking $X\in G\setminus F$, we have that
$X,A\setminus X\in G$, contradicting the fact that $G$ is centred. Hence
$F$ is maximal among centred families.

(3) This follows immediately from (1) and (2).

(4) Clearly, an ultrafilter is a maximal filter. On the other hand, if
$F$ is a maximal filter, then by (1) $F$ can be extended to a maximal
centred family $G$. But $G$ is an ultrafilter by (2) so that $F=G$
by the maximality of $F$.

(5) This follows from (3) and remark (4) above.
\end{proof}

We now give the construction of an ultraproduct of left $S$-acts; we
could, of course, define ultraproducts for any class of interpretations
or structures
of a given first order language $L$, that is, any class
of $L$-{\em structures}, but we prefer to be specific and leave the
reader
to extrapolate. A point of notation: for any arbitrary language $L$ we
make a distinction between an $L$-structure $\mathbf{A}$ and
the underlying set $A$ of $\mathbf A$, whereas for $S$-acts we do not.

If $B=\prod_{i\in I}B_i$ is a product of left $S$-acts $B_i$, $i\in
I$,
and $\Phi$ is an ultrafilter on $I$, then we
define a relation $\equiv_{\Phi}$ on $B$ by the rule that
\[(a_i)\equiv_{\Phi} (b_i)\mbox{ if and only if }
\{ i\in I\mid a_i=b_i\} \in \Phi.\] It is a fact that
$\equiv_{\Phi}$ is an equivalence and moreover an $S$-act
congruence, so that putting
\[U=(\prod_{i\in I}B_i)/{\Phi}=(\prod\limits_{i\in I}B_i)/\equiv_{\Phi}\]
and denoting the $\equiv_{\Phi}$-class of $f\in \prod\limits_{i\in I}B_i$
by $f/\Phi$, $U$
is an $S$-act under the operation
\[s\, (a_i)/\Phi=(sa_i)/\Phi,\]
which we call the {\em ultraproduct} of
left $S$-acts $B_i$, $i\in I$, under the ultrafilter $\Phi$.

 Ultraproducts are of central importance to us, due to the
celebrated theorem of  \L os.

\begin{theorem}\label{ultra}\cite{changkeisler} Let $L$ be a first order
language and $\mathcal{A}$ a class of $L$-structures. If
$\mathcal{A}$ is axiomatisable, then $\mathcal{A}$ is closed under
the formation of ultraproducts.
\end{theorem}

Let $\varkappa$ be an infinite cardinal;
thus $\varkappa$ is a limit ordinal and  we may regard
$\varkappa$ as the union of all smaller ordinals.
A filter $F$ is called \em $\varkappa$-regular \em if there exists a
family $R=\{S_\alpha\mid\alpha\in\varkappa\}$
of distinct elements of $F$ such that any intersection of any
 infinite subset of the
family $R$ is empty.

\begin{proposition}\label{filter2}
\label{regultrafil} For any infinite cardinal $\varkappa$ there
exists an $\varkappa$-regular ultrafilter on set $I$, where
$|I|=\varkappa$.
\end{proposition}
\begin{proof} According to Proposition~\ref{filter1}
 it is enough to construct a set
$I$, $|I|=\varkappa$, and a centred family
$F=\{S_\alpha\mid\alpha\in\varkappa\}$
of distinct subsets of $I$ such that
the intersection of any infinite subset of $F$
is empty.

Consider the set $I=\{v\mid v\subseteq\varkappa,\mbox{$v$ is finite}
\}$. Clearly the cardinality of $I$ is equal to $\varkappa$. Let
$S_\alpha=\{v\mid v\in I,\alpha\in v\}$ and
$R=\{S_\alpha\mid\alpha\in\varkappa\}$. If $S_{\alpha_1},
\ldots,S_{\alpha_n}\in R$ then
\[\{ \alpha_1,\hdots ,\alpha_n\}\in S_{\alpha_1}\cap\ldots\cap S_{\alpha_n}.\]
 Therefore the
family $R$ is centred and the intersection of any infinite collection
of its
elements is empty.
\end{proof}

We require one further technical result
concerning ultraproducts that will be
needed for later sections.

\begin{theorem}\label{filter3}\label{filtprod} Let $F$ be an $\varkappa$-regular
filter on a set
$I$, $|I|=\varkappa$,
and let  ${A}_i$ ($i\in I$) be $S$-acts
of infinite cardinality $\lambda$. Then the cardinality of the
filtered product $(\prod_{i\in I}{A}_i)/F$ is equal to
$\lambda^\varkappa$.
\end{theorem}
\begin{proof} Let us denote the set of all finite sequences of
elements of the cardinal $\lambda$ by $\lambda^{<\omega}$. It is
clear that the cardinality of the set $\lambda^{<\omega}$ coincides
with $\lambda$. Since the cardinality of the filtered products does
not depend on the
fact that the ${A}_i$'s are $S$-acts, but merely on their
cardinality,  we can suppose that
${A}_i=\lambda^{<\omega}$ for all $i\in I$. Since
$|(\lambda^{<\omega})^I|=\lambda^\varkappa$ then
$|(\lambda^{<\omega})^I/F|\leqslant\lambda^\varkappa$. Thus it is
enough to construct an embedding $\phi$ of the set
$\lambda^\varkappa=\{f\mid f:\varkappa\rightarrow\lambda\}$ in the
set $(\lambda^{<\omega})^I/F$.

Let the family $R=\{S_\alpha\mid\alpha\in\varkappa\}$ of elements from
the filter $F$ satisfy the required condition from the definition of
$\varkappa$-regular filter, i.e. any intersection of an infinite
subset of the family $R$ is empty; clearly,
 for any $i\in I$, we must have
that $\{ \alpha:i\in S_{\alpha}\}$ is finite. We may add the set
$I\setminus\bigcup R$ to the set $S_0$, so we can consider that
$I=\bigcup R$. Let $f:\varkappa\rightarrow\lambda$ be an arbitrary
map. Define the map $f^*:I\rightarrow\lambda^{<\omega}$ in the
following way: $f^*(i)=\langle
f(\alpha_1),\dots,f(\alpha_n)\rangle$, where $i\in I$,
$\{\alpha_1,\ldots,\alpha_n\}=\{\alpha\mid i\in S_\alpha\}$ and
$\alpha_i<\alpha_{i+1}$ $(1\leq i\leq n-1)$. We consider $
\phi(f)=f^*/F$. Let $f_1,f_2\in\lambda^\varkappa$ with $f_1\neq
f_2$, so that $f_1(\alpha)\ne f_2(\alpha)$ for some
$\alpha\in\varkappa$.  Then $f_1^*(i)\ne f_2^*(i)$ for any $i\in
S_\alpha$; since $S_\alpha\in F$ we have that  $f_1^*/F\ne f_2^*/F$
and so  $\phi$ is one-one as required.

\end{proof}

\section{Axiomatisability of $\mathcal{WF}$}\label{wfsection}

We begin our considerations of axiomatisability with the class
$\mathcal{WF}$; the results of this section are all taken from
\cite{BG}.

 At this point it is useful to give further details on
tensor products.

\begin{lemma} \label{tossing}\cite{kilp}
Let $A$ be a right $S$-act and $B$ a left $S$-act.
Then for $a,a'\in A$ and $b,b'\in B$,
$a\otimes b=a'\otimes b'$ if and only if there exist
$s_1,t_1,s_2,t_2,\hdots ,s_m,t_m\in S$, $a_2,\hdots ,a_m\in A$
and $b_1,\hdots ,b_m\in B$ such that
\[%
\begin{array}
{rclrcl}
&&  & b&=&s_1b_1\\
as_{1}&=&a_2t_{1}   & t_{1}b_1&=&s_{2}b_2\\
a_2s_{2}&=&a_3t_{2}   & t_{2}b_2&=&s_{3}b_3\\
&\vdots&&& \vdots & \\
a_ms_m&=&a't_{m}   & t_{m}b_m&=&b'.
\end{array}
\]
\end{lemma}

The sequence presented  in Lemma \ref{tossing} will be called a {\em tossing}
(or scheme) $\tee$
 of length $m$ over $A$ and $B$ connecting
$(a,b)$ to $(a',b')$. The {\em skeleton} of $\tee$,
$\s=\s(\tee)$, is the sequence
\[\s=(s_1,t_1,\hdots ,s_m,t_m)\in S^{2m}.\]
The set of all skeletons is denoted by $\mathbb{S}$. By considering
trivial acts it is easy to
see that $\mathbb{S}$ consists of all sequences of elements of $S$
of even length.

Let $a,a'\in S$
and let $\mathcal{S}=(s_{1},t_{1},...,s_{m},t_{m})$
be a skeleton of length $m$. We say that
the triple $(a,\mathcal{S},a')$ is
{\em realised} if there are elements $a_2,a_3,\hdots ,a_m\in S$
such that
\[
\begin{array}{rcl}
as_{1}&=&a_{2}t_{1}\\
a_{2}s_{2}& =&a_{3}t_{2}\\
& \vdots & \\
a_{m}s_{m}&=&a't_{m}.
\end{array}
\]
We let $\mathbb{T}$ denote the set of realised triples.

A realised triple $(a,\mathcal{S},a')$ is
{\em witnessed by} $b,b_1,\hdots, b_m,b'$ where
$b,b_1,\hdots ,$\\ $ b_m,b'$ are elements of a left $S$-act $B$
if
\[\begin{array}{rcl}
b&=&s_1b_1\\
t_1b_1&=&s_2b_2\\
&\vdots &\\t_mb_m&=&b'.\end{array}\]

We know that
if $a,a'\in A$ and $b,b'\in B$,
where $A$ is a right $S$-act and $B$ a left
$S$-act,
then $a\otimes b=a^{\prime}\otimes b^{\prime}$ in $A\otimes B$
 if and only if there exists a tossing
  $\mathcal{T}$ from $(a,b)$ to $(a^{\prime
},b^{\prime})$ over $A$ and $B$, with skeleton $\mathcal{S},$ say.
If the equality $a\otimes b=a^{\prime}\otimes
b^{\prime}$ holds also in $(aS\cup a^{\prime}S)\otimes B,$ and is
determined by some tossing $\mathcal{T'}$ from $(a,b)$ to $(a^{\prime
},b^{\prime})$ over $aS\cup a^{\prime}S$ and $B$
 with skeleton $\s'=\mathcal{S(T')}$
then we say that
$\mathcal{T}'$ is a
{\em replacement tossing} for $\mathcal{T}$,
 $\s'$ is  a  \emph{replacement skeleton
for } $\mathcal{S}$ and (in case $A=S$)
triples $(a,\s',a')$ will be called
{\em replacement triples} for $(a,\s,a')$.

Note that, for any left $S$-act $B$, right $S$-act $A$ and
$(a,b),(a',b')\in A\times B$, if
$\mathcal{S}=(s_{1},t_{1},...,s_{m},t_{m})$ is the skeleton of a
tossing from
$(a,b)$ to $(a^{\prime},b^{\prime})$ over $A$ and $B$, then
$\gamma_{\mathcal{S}}(b,b')$ holds in $B$, where
$\gamma_{\mathcal{S}}$ is the  formula
\[\gamma_{\mathcal{S}}(y,y^{\prime})
\leftrightharpoons(\exists y_{1})(\exists y_{2}%
)\cdots(\exists y_{m})(y=s_{1}y_{1}\wedge t_{1}y_{1}=s_{2}y_{2}\wedge
\cdots\wedge t_{m}y_{m}=y^{\prime}).\]
For any $\s\in\mathbb{S}$ we define $\psi_{\mathcal{S}}$ to be the
sentence
\[\psi_{\mathcal{S}}\leftrightharpoons
(\forall y)(\forall y^{\prime})\neg\gamma_{\mathcal{S}}(y,y^{\prime}).\]

\begin{theorem}\label{wf}\cite{BG} The following conditions are equivalent for a monoid
$S$:

(1) the class $\mathcal{WF}$ is axiomatisable;

(2) the class $\mathcal{WF}$ is closed under ultraproducts;

(3) for every skeleton $\s$ over $S$ and $a,a'\in S$
there exist finitely many skeletons
$\s_1,\hdots ,\s_{\alpha(a,\mathcal{S},a')}$ over $S$,
such that for any weakly flat left
$S$-act $B$, if $(a,b),(a',b')\in S\times B$ are
connected by a tossing $\tee$ over $S$ and $B$ with $\s(\tee)=\s$,
then $(a,b)$ and $(a',b')$ are connected by a tossing
$\tee'$ over $aS\cup a'S$ and $B$ such that $\s(\tee')=\s_k$,
for some $k\in\{ 1\,\hdots ,\alpha(a,\mathcal{S},a')\}$.
\end{theorem}
\begin{proof} That (1) implies (2)
follows from Theorem~\ref{ultra} (\L os's theorem).

Suppose now that (2) holds but that (3) is false.
Let $J$ denote the set of all finite subsets of  $\mathbb{S}$ and
suppose that,
for some skeleton $\s_0=(s_1,t_1,\hdots ,s_m,t_m)\in\mathbb{S}$
and $a,a'\in S$, for
every $f\in J$, there is a weakly flat left $S$-act
$B_f$ and $b_f,b_f'\in B_f$ with the
pairs $(a,b_{f}),(a',b_f')\in S\times B_{f}$
connected by a tossing $\mathcal{T}_{f}$ with skeleton
$\mathcal{S}_{0}$,
 but such that
no tossing over $aS\cup a'S$ and
$B_{f}$ connecting $(a,b_{f})$ and $(a',b_f')$ has a
skeleton belonging to the set $f$.

For each $\mathcal{S}\in\mathbb{S}$ let $J_{\mathcal{S}}=\left\{  f\in
J:\mathcal{S}\in f\right\}  .$ Each intersection of finitely many of the sets
$J_{\mathcal{S}}$ is non-empty (because $\mathbb{S}$ is infinite), so there
exists an ultrafilter $\Phi$ over $J$ such that each $J_{\mathcal{S}}$
($\mathcal{S}\in\mathbb{S}$) belongs to $\Phi.$
Notice that
$a\otimes\underline{b}
=a'\otimes\underline{b'}$
in $S\otimes B$
where $B=\prod_{f\in J}B_{f}$, $\underline{b}(f)=b_f$ and
$\underline{b'}(f)=b'_f$,
and that this equality is determined by a tossing over
$S$ and $B$ having skeleton
$\mathcal{S}_{0}.$
 It follows that the equality $a\otimes
(\underline{b}\, /\Phi)=a'\otimes(\underline{b'}\, /\Phi)$
holds also in $S\otimes U$,
 where $U=(\prod_{f\in J}B_{f})/{\Phi},$ and
is also
determined by a tossing over $S$ and $U$
 with skeleton $\mathcal{S}_{0}.$

By our assumption,
 $U$ is weakly flat,
so that $(a,\underline{b}\, /\Phi)$
and
$(a',\underline{b'}\, /\Phi)$
are connected via a replacement tossing $\tee'$ over
$aS\cup a'S$ and $U$, say
\[%
\begin{array}
{rclrcl}%
& & & \underline{b}\, /\Phi&=&u_{1}\underline{d_1}\, /\Phi\\
au_{1}&=&c_2v_{1}   & v_{1}\underline{d_1}\, /\Phi&=&u_{2}%
\underline{d_2}\,/\Phi\\
c_2u_{2}&=&c_3v_{2}   & v_{2}\underline{d_2}\, /\Phi&=&u_{3}
\underline{d_3}\, /\Phi\\
&\vdots&&& \vdots&  \\
c_nu_{n}&=&a'v_{n}   & v_{n}\underline{d_n}\, /\Phi&
=&\underline{b'}\, /\Phi,
\end{array}
\]
where $\underline{d_i}(f)=d_{i,f}$ for any $f\in J$ and
$i\in\{ 1,\hdots ,n\}$. We put $\s'=\s(\tee')$.

 As $\Phi$ is closed under finite intersections, there exists $D\in
\Phi$ such that
\[
\begin{array}
{rclrcl}%
&& & b_{f}&=&u_{1}d_{1,f}\\
au_{1}&=&c_2v_{1}   & v_{1}d_{1,f}&=&u_{2}d_{2,f}\\
c_2u_{2}&=&c_3v_{2}   & v_{2}d_{2,f}&=&u_{3}d_{3,f}\\
& \vdots&&&\vdots & \\
c_nu_{n}&=&a'v_{n} &   v_{n}d_{n,f}&=&b_{f}^{\prime}%
\end{array}\]
%\end{equation}
whenever $f\in D.$
 Now, suppose $f\in D\cap J_{\mathcal{S}^{\prime}}.$ Then,
from the tossing just considered,
we see that $\mathcal{S}^{\prime}$ is a
replacement skeleton for skeleton $\mathcal{S}_{0},$
connecting $(a,b_f)$ and $(a',b_f')$ over $aS\cup a'S$
and $B_f$. But $\s'\in f$, contradicting the choice of
$(a,b_f)$ and $(a',b_f')$. Thus (3) holds.

Finally, suppose that  (3) holds.
  We  introduce a sentence
corresponding to each element of $\mathbb{T}$
 in such a way that the resulting
set of sentences axiomatises the class $\mathcal{WF}.$

We let $\mathbb{T}_1$ be the set of realised triples that
are not witnessed in any weakly flat left $S$-act $B$, and put
$\mathbb{T}_2=\mathbb{T}\setminus\mathbb{T}_1$.
For $T=(a,\s,a')\in\mathbb{T}_1$ we let
$\psi_T$ be the sentence $\psi_{\mathcal{S}}$ defined before
the statement of this theorem. If $T=(a,\s,a')\in\mathbb{T}_2$, then
$\s$ is the skeleton of some scheme joining $(a,b)$
to
$(a',b')$ over $S$ and some weakly flat left $S$-act $B$. By our
assumption (3), there is a finite list of replacement
skeletons
 $\mathcal{S}_{1},...,\mathcal{S}_{\alpha(T)}$. Then,
for each $k\in\left\{  1,...,\alpha(T)\right\}  ,$
 if $\mathcal{S}%
_{k}=(u_{1},v_{1},...,u_{h},v_{h}),$ there exist a weakly flat
left $S$-act $C_k$ elements
$c,c^{\prime},c_{1},...,c_{h}\in C_{k}$, and elements
$q_{2},...,q_{h}\in aS\cup a'S$ such that

\begin{equation}%
\begin{array}
{rclrcl}%
& && c&=&u_{1}c_1\\
au_{1}&=&q_2v_{1}  & v_{1}c_1&=&u_{2}c_2\\
q_2u_{2}&=&q_3v_{2}   & v_{2}c_2&=&u_{3}c_3\\
&\vdots&&& \vdots&  \\
q_hu_{h}&=&a'v_{h}   & v_{h}c_h&=&c'.%
\end{array}
\tag{i}%
\end{equation}
For each $k,$ we fix such a list $q_{2},...,q_{h}$ of elements, for future
reference, and let
$\varphi_T$ be the sentence
\[\varphi_T\leftrightharpoons (\forall y)(\forall y^{\prime})(\gamma_{\mathcal{S}%
}(y,y^{\prime})\rightarrow\gamma_{\mathcal{S}_{1}}(y,y^{\prime})\vee\cdots
\vee\gamma_{\mathcal{S}_{\alpha(T)}}(y,y^{\prime})).
\]
 Let
\[\Sigma_{\mathcal{WF}}=\left\{  \psi_T:T\in\mathbb{T}_1\right\}\cup
\left\{ \varphi_T:T\in\mathbb{T}_2\right\}.\] We claim that
 $\Sigma_{\mathcal{WF}}$ axiomatises $\mathcal{WF}$.

Suppose first that $D$ is any weakly flat left $S$-act.
Let $T\in\mathbb{T}_1$. Then $T=(a,\s,a')$ is
a realised triple. Since $T$ is not witnessed in {\em any} weakly
flat left $S$-act, $T$ is certainly not witnessed in $D$
so that $D\vDash \psi_T$.

On the other hand, for
 $T=(a,\s,a')\in\mathbb{T}_{2},$
where $\s=(s_1,t_1,\hdots ,s_m,t_m)$, if $d,d'\in D$
are such that $\gamma_{\mathcal{S}}(d,d')$ is true,
then there are elements $d_1,\hdots d_m\in D$ such that
\[\begin{array}{rcl} d&=&s_1d_1\\
t_1d_1&=&s_2d_2\\
&\vdots&\\
t_md_m&=&d,'\end{array}\]
which together with the fact that $T$ is a realised triple,
gives that $(a,d)$ is connected to $(a',d')$ over $S$ and
$D$ via a tossing with skeleton $\s$. Because $D$ is weakly flat,
$(a,d)$ and $(a',d')$ are connected
over $aS\cup a'S$ and $D$, and by assumption (3),
we can take the tossing to have skeleton one of
$\s_1,\hdots ,\s_{\alpha(T)}$, say $\s_k$. Thus
$D\vDash \gamma_{\mathcal{S}_k}(d,d')$ and it follows that
$D\vDash \varphi_T$. Hence $D$ is a model of $\Sigma_{\mathcal{WF}}$.

Conversely, we show that every model of $\Sigma_{\mathcal{WF}}$ is weakly flat.
Let $C\vDash \Sigma_{\mathcal{WF}}$ and suppose that $a,a'\in S,c,c'\in C$ and we
have a tossing
\begin{equation*}%
\begin{array}
{rclrcl}%
&& & c&=&s_{1}c_1\\
as_{1}&=&a_2t_{1}   & t_{1}c_1&=&s_{2}c_2\\
a_2s_{2}&=&a_3t_{2}   & t_{2}c_2&=&s_{3}c_3\\
&\vdots&&& \vdots & \\
a_ms_{m}&=&a't_{m}   & t_{m}c_m&=&c'%
\end{array}
\end{equation*}
with skeleton $\s=(s_1,t_1,\hdots ,s_m,t_m)$
over $S$ and $C$. Then the triple $T=(a,\s,a')$ is
realised, so that $T\in\mathbb{T}$. Since
$\gamma_{\mathcal{S}}(c,c')$ holds, $C$
cannot be a model of  $\psi_{\mathcal{T}}$. Since
$C\vDash\Sigma_{\mathcal{WF}}$ it follows that $T\in\mathbb{T}_2$.
But then $\varphi_T$ holds in $C$ so that
for some $k\in\{ 1\hdots ,\alpha(T)\}$ we have
that $\gamma_{\mathcal{S}_k}(c,c')$ is true. Together with
the equalities on the left hand side of (i) we have a tossing
over $aS\cup a'S$ and $C$ connecting $(a,c)$ to $(a',c')$.
Thus $C$ is weakly flat.
\end{proof}

\section{Axiomatisability of $\mathcal{F}$}

We now turn our attention to the class $\mathcal{F}$ of flat left
$S$-acts. The results of this section are again taken from \cite{BG}. First we consider
 the {\em finitely presented flatness lemma} of \cite{BG}, which is crucial to
our
arguments.

Let $\mathcal{S}=(s_1,t_1,\hdots ,s_m,t_m)\in\mathbb{S}$
be a skeleton. We let $F^{m+1}$ be the free right $S$-act
\[xS\amalg x_{2}S\amalg...\amalg x_{m}S\amalg
x^{\prime}S\]
and let $\rho_{\mathcal{S}}$ be the congruence on $F^{m+1}$
generated
by
\[\left\{  (xs_{1},x_{2}t_{1}),(x_{2}s_{2},x_{3}t_{2}),...,(x_{m-1}s_{m-1}%
,x_{m}t_{m-1}),(x_{m}s_{m},x^{\prime}t_{m})\right\} .\]
We denote the $\rho_{\mathcal{S}}$-class
of $w\in F^{m+1}$ by $[w]$.
If $B$ is a left $S$-act and $b,b_1,\hdots ,b_m,b'\in B$
are such that
\[b=s_1b_1,t_1b_1=s_2b_2,\hdots, t_mb_m=b'\]
then the tossing
\[%
\begin{array}
[c]{rclrcl}%
& & & b&=&s_{1}b_{1}\\
\left[ x\right]s_{1}&=&\left [x_2\right]t_{1} &   t_{1}b_{1}&=&s_{2}b_{2}\\
\left[ x_2\right]s_{2}&=&\left[ x_3\right]t_{2}   & t_{2}b_{2}&=&s_{3}b_{3}\\
&\vdots &&  & \vdots&\\
\left[ x_{m-1}\right]s_{m-1}&=&\left[ x_m\right]t_{m-1}   & t_{m-1}b_{m-1}&=&s_{m}b_{m}\\
\left[ x_m\right]s_{m}&=&\left[ x^{\prime}\right]t_{m}   & t_{m}b_{m}&=&b^{\prime}.
\end{array}
\]
over $F^{m+1}/\rho_{\mathcal{S}}$ and $B$ is called a {\em standard tossing}
with skeleton $\mathcal{S}$ connecting $([x],b)$ to $([x'],b')$.

We refer the reader to \cite{BG} for the proof of the following lemma.

\begin{lemma}
\label{fplemma}\cite{BG}
%\index{fplemma}
  The following are
equivalent for a left $S$-act $B$:

(1) $B$ is flat;

(2) $-\otimes B$ preserves all embeddings
of $A$ in $C$,  where
$A$ is a finitely generated subact
of a finitely presented right $S$-act $C$;

(3) $-\otimes B$ preserves the embedding
of $\left[  x\right] S\cup\left[  x^{\prime}\right]  S$
into $F^{m+1}/\rho_{\mathcal{S}}$, for
all skeletons $\mathcal{S}$;

(4) if $([x],b)$ and $([x'],b')$ are connected by a
standard
tossing
over $F^{m+1}/\rho_{\mathcal{S}}$ and $B$ with skeleton $\mathcal{S}$, then
they are connected by a tossing over
$\left[  x\right] S\cup\left[  x^{\prime}\right] S$  and $B$.
\end{lemma}

The construction of $F^{m+1}/\rho_{\mathcal{S}}$ enables us
to observe that for any left $S$-act $B$ and any $b,b^{\prime}\in B$,
 a skeleton
$\mathcal{S}=(s_{1},t_{1},...,s_{m},t_{m})$ is the skeleton of a tossing from
$(a,b)$ to $(a^{\prime},b^{\prime})$ over $A$ and $B$ for some
 $A$ and
some $a,a^{\prime}\in A$ if and only if
$\gamma_{\mathcal{S}}(b,b')$ holds in $B$, where
$\gamma_{\mathcal{S}}$ is the sentence defined before
Theorem~\ref{wf}. We also note that if $(a,b),(a',b')$
are connected via a tossing with skeleton $\s$, then
$([x],b),([x'],b)\in F^{m+1}
/\rho_{\mathcal{S}}$ are connected via the standard tossing with skeleton $\s$.

\begin{theorem}\label{flatsareaxiomatisable}\cite{BG}
 The following conditions are equivalent for a monoid
$S$:

(1) the class $\mathcal{F}$ is axiomatisable;

(2) the class $\mathcal{F}$ is closed under formation of ultraproducts;

(3) for every skeleton $\mathcal{S}$ over $S$ there exist finitely many
replacement
skeletons $\mathcal{S}_{1},...,\mathcal{S}_{\alpha(\mathcal{S})}$
over $S$ such that, for any
right $S$-act $A$ and any flat act
left $S$-act $B$, if $(a,b),(a^{\prime},b^{\prime})\in
A\times B$ are connected by a tossing $\mathcal{T}$ over $A$ and $B$ with
$\mathcal{S(T)=S},$ then $(a,b)\,$and $(a^{\prime},b^{\prime})$ are connected
by a tossing $\mathcal{T}^{\prime}$ over $aS\cup a^{\prime}S$ and $B$ such
that $\mathcal{S(T}^{\prime}\mathcal{)=S}_{k},$ for some $k\in\left\{
1,...,\alpha(\mathcal{S})\right\}  .$
\end{theorem}
\begin{proof}
The implication (1) implies (2) is clear from Theorem~\ref{ultra}.

The proof of (2) implies (3) follows the pattern set by that
of Theorem~\ref{wf}, in particular, $J$, the sets
$J_{\mathcal{S}}$ for $\s\in\mathbb{S}$ and the ultrafilter
$\Phi$ are defined as in that theorem.

Suppose that $\mathcal{F}$ is
closed under formation of ultraproducts,
but that assertion (3) is false. Let
 $J\ $denote the
family of all finite subsets of $\mathbb{S}.$ Suppose $\mathcal{S}_{0}%
=(s_{1},t_{1},...,s_{m},t_{m})\in\mathbb{S}$ is such that, for every $f\in J$,
there exist
a right $S$-act $A_f$,
a flat left $S$-act $B_{f}$, and
pairs $(a_{f},b_{f}),$ $(a_{f}^{\prime},b_{f}^{\prime})\in A_{f}\times B_{f}$
such that $(a_{f},b_{f})$ and $(a_{f}^{\prime},b_{f}^{\prime})$ are connected
over $A_{f}$ and $B_{f}$ by a tossing $\mathcal{T}_{f}$ with skeleton
$\mathcal{S}_{0},$ but no
replacement tossing over $a_{f}S\cup a_{f}^{\prime}S$ and
$B_{f}$ connecting $(a_{f},b_{f})$ and $(a_{f}^{\prime},b_{f}^{\prime})$ has a
skeleton belonging to the set $f.$
 Note that $\underline
{a}\otimes\underline{b}=\underline{a}^{\prime}\otimes\underline{b}^{\prime}$
in $A\otimes B,$ where
for each $f\in J$,
$\underline{a}(f)=a_f,\underline{b}(f)=b_f$, $A=\prod_{f\in J}A_{f}$ and $B=\prod_{f\in J}B_{f},$ and that this equality is determined by a tossing over
$A$ and $B$ (the ``product'' of the $\mathcal{T}_{f}$'s) having skeleton
$\mathcal{S}_{0}.$ If follows that the equality
 $\underline{a}\otimes
(\underline{b}\, /\Phi)=
\underline{a}^{\prime}\otimes(\underline{b'}\, /\Phi)$
holds also in $A\otimes U,$ where
$U=(\prod_{f\in J}B_{f})/{\Phi},$ and
is determined by a tossing over $A$ and $U$ with
skeleton $\mathcal{S}_{0}.$
Because $U$ is flat, $\mathcal{S}_{0}$ has a replacement skeleton
$\mathcal{S(T}^{\prime})=\mathcal{S}^{\prime}=(u_{1},v_{1},...,u_{n},v_{n})$,
where $\mathcal{T}^{\prime}$ is a tossing
\[%
\begin{array}
[c]{rclrcl}
&&  & \underline{b}\, /\Phi&=&u_{1}\underline{d_1}\, /\Phi\\
\underline{a}u_{1}&=&\underline{c_2}v_{1}
& v_{1}\underline{d_1}\, /\Phi&=&u_{2}\underline{d_2}\, /\Phi\\
\underline{c_2}u_{2}&=&\underline{c_3}v_{2}
& v_{2}\underline{d_2}\, /\Phi&=&u_{3}\underline{d_3}\, /\Phi\\
&\vdots& &&\vdots& \\
\underline{c_n}u_{n}&=&\underline{a}^{\prime}v_{n}
& v_{n}\underline{d_n}\, /\Phi
&=&\underline{b'}\, /\Phi%
\end{array}
\]
where for $2\leq j\leq n$ and $f\in J$ we have
$\underline{c_j}(f)=c_{j,f}\in a_{f}S\cup a_{f}^{\prime}S\ $ and
for $1\leq j\leq n$ and $f\in J$ we have
 $\underline{d_j}(f)=d_{j,f}\in
B_{f}.$ As $\Phi$ is closed under finite intersections, there exists $D\in
\Phi$ such that
\[
\begin{array}
[c]{rclrcl}
&&  & b_f&=&u_{1}d_{1,f}\\
a_fu_{1}&=&c_{2,f}v_{1}  & v_{1}d_{1,f}&=&u_{2}%
d_{2,f}\\
c_{2,f}u_{2}&=&c_{3,f}v_{2}   & v_{2}d_{2,f}&=&u_{3}d_{3,f}\\
&\vdots& &&\vdots& \\
c_{n,f}u_{n}&=&a_f'v_{n}   & v_{n}d_{n,f}
&=&b_f'
\end{array}\]
whenever $f\in D.$ Now, suppose $f\in D\cap J_{\mathcal{S}^{\prime}}.$ Then,
from the tossing just considered, we see that $\mathcal{S}^{\prime}$ is a
replacement skeleton for skeleton $\mathcal{S}_{0},$ the latter being the
skeleton of tossing $\mathcal{T}_{f}$ connecting the pairs $(a_{f,}b_{f})$ and
$(a_{f}^{\prime},b_{f}^{\prime})$ over $A_{f}$ and $B_{f}.$ But because
$\mathcal{S}^{\prime}$ belongs to $f,$ this is impossible.
This completes the proof that (2) implies (3).

Finally, we show that (3) implies (1). Suppose every skeleton
requires only finitely many replacement skeletons, as made precise in the
statement of (3) above. We aim to use this condition to construct a set of
axioms for $\mathcal{F}.$

 Let $\mathbb{S}_{1}$ denote
the set of all elements of $\mathbb{S}$ that are \emph{not} the skeleton
of
any tossing connecting two elements of $A\times B,$ where $A$ ranges over all
right $S$-acts and $B$ over all flat left $S$-acts, and let $\mathbb{S}%
_{2}=\mathbb{S}\backslash\mathbb{S}_{1}.$

For $\mathcal{S}\in\mathbb{S}_{2},$ the comments preceding the theorem
yield that $\mathcal{S}$ is the
skeleton of
a standard tossing joining $(\left[  x\right]  ,b)$
 to $(\left[  x^{\prime}\right] ,b^{\prime})$
 over $F^{m+1}/\rho_{\mathcal{S}}$ and $B$ where $B$ is flat,
$b,b^{\prime}\in B,$
and $F^{m+1}$ and $\rho_{\mathcal{S}}$
are as defined in Lemma \ref{fplemma}.

Let $\mathcal{S}_{1},...,\mathcal{S}_{\alpha(\mathcal{S})}$ be a set of
replacement skeletons for $\mathcal{S}$ as provided by assertion (3)
and
without loss of generality
suppose that replacements for standard tossings may be
chosen to have skeletons
from $\{ \mathcal{S}_1,\hdots ,\mathcal{S}_{\alpha'(\mathcal{S})}\}$,
where $\alpha'(\mathcal{S})\leq \alpha(\mathcal{S})$. Hence
for each $k\in\left\{  1,...,\alpha'(\mathcal{S})\right\}$ if
 $\mathcal{S}%
_{k}=(u_{1},v_{1},...,u_{h},v_{h})$, there exist a flat left
$S$-act  $C_k$, elements
$c,c^{\prime},c_{1},...,c_{h}\in C_{k}$, and elements $p_{2},...,p_{h}%
\in\left[  x\right] S\cup\left[  x^{\prime}\right] S$ such that
\begin{equation}%
\begin{array}
[c]{rclrcl}
& && c&=&u_{1}c_{1}\\
\left[  x\right]  u_{1}&=&p_{2}v_{1}
  & v_{1}c_{1}%
&=&u_{2}c_{2}\\
p_{2}u_{2}&=&p_{3}v_{2}   & v_{2}c_{2}&=&u_{3}c_{3}\\
& \vdots& &&\vdots& \\
p_{h}u_{h}&=&\left[  x^{\prime}\right]  v_{h} &
v_{h}c_{h}&=&c^{\prime}.
\end{array}
\tag{ii}%
\end{equation}
For each $k,$ we fix such a list $p_{2},...,p_{h}$ of elements, for future
reference, and define $\varphi_{\mathcal{S}}$ to be the sentence
\[\varphi_{\mathcal{S}}
\leftrightharpoons(\forall y)(\forall y^{\prime})(\gamma_{\mathcal{S}%
}(y,y^{\prime})\rightarrow\gamma_{\mathcal{S}_{1}}(y,y^{\prime})\vee\cdots
\vee\gamma_{\mathcal{S}_{\alpha'(\mathcal{S})}}(y,y^{\prime})).
\]
Let
\[\Sigma_{\mathcal F}=\left\{  \psi
_{\mathcal{S}}:\mathcal{S}\in\mathbb{S}_{1}\right\}  \cup\left\{
\varphi_{\mathcal{S}}:\mathcal{S}\in\mathbb{S}_{2}\right\}  .\]
 We claim that $\Sigma_{\mathcal F}$ axiomatises
$\mathcal{F}$.

Suppose first that $D$ is any flat left $S$-act.

For $\mathcal{S}\in\mathbb{S}_{1},$ if $D$ did not satisfy $\psi
_{\mathcal{S}},$ then we would have $\gamma_{\mathcal{S}}(d,d^{\prime})$ for
some $d,d^{\prime}\in D,$ and so, by the comments
preceding the statement of this theorem, $\mathcal{S}$ is the skeleton of some tossing joining $(a,d)$ to
$(a^{\prime},d^{\prime})$ over some right $S$-act
$A$ and flat left $S$-act $D,$ contrary to the
fact that $\mathcal{S}\in\mathbb{S}_{1}.$ Therefore, $D\vDash
\psi_{\mathcal{S}}.$

Now take any $\mathcal{S}\in\mathbb{S}_{2},$ and suppose $d,d^{\prime}\in D$
are such that $D$ satisfies $\gamma_{\mathcal{S}}(d,d^{\prime}).$ Then, as
noted earlier, $(\left[  x\right]  ,d)$ and $(\left[
x^{\prime}\right]  ,d^{\prime})$ are joined over
$F^{m+1}/\rho_{\mathcal{S}}$ and $D$ by a  standard tossing with skeleton $\mathcal{S},$ and
therefore, by assumption, by a tossing over $\left[  x\right]
S\cup\left[  x^{\prime}\right]  S$ and $D$
with skeleton $\mathcal{S}_{k}$ for some $k\in\left\{  1,...,\alpha'
(\mathcal{S})\right\}  .$ It is now clear that $\gamma_{\mathcal{S}_{k}%
}(d,d^{\prime})$ holds in $D,$ as required. We have now shown that
 $D\vDash\Sigma_{\mathcal F}.$

Finally, we show that a left $S$-act $C$ that satisfies
$\Sigma_{\mathcal F}$ must be
flat.
We need only show that condition (4) of Lemma \ref{fplemma}
holds for $C$.
Let $\s\in\mathbb{S}$ and
suppose we have a standard tossing
\begin{equation}%
\begin{array}
[c]{rclrcl}%
&  && c&=&s_{1}c_{1}\\
\left[  x\right] s_{1}&=&\left[  x_{2}\right] t_{1}   & t_{1}c_{1}&=&s_{2}c_{2}\\
&\vdots&& &\vdots & \\
\left[  x_{m}\right]  s_{m}&=&\left[  x^{\prime}\right] t_{m}   & t_{m}c_{m}&=&c^{\prime}%
\end{array}
\tag{iii}%
\end{equation}
over $F^{m+1}/\rho_{\mathcal{S}}$ and $C.$ If $\mathcal{S}$ belonged to
$\mathbb{S}_{1},$ then $C$ would satisfy the sentence $(\forall y)(\forall
y^{\prime})\lnot\gamma_{\mathcal{S}}(y,y^{\prime}),$ and so $\lnot
\gamma_{\mathcal{S}}(c,c^{\prime})$ would hold, contrary to the sequence of
equalities in the right-hand column of (ii). Therefore, $\mathcal{S}\ $belongs
to $\mathbb{S}_{2}.$ Because $C$ satisfies $\varphi_{\mathcal{S}}$ and
because $\gamma_{\mathcal{S}}(c,c^{\prime})$ holds, it follows that
$\gamma_{\mathcal{S}_{k}}(c,c^{\prime})$ holds for some $k\in\left\{
1,...,\alpha'(\mathcal{S})\right\}  .$ If $\mathcal{S}_{k}=(u_{1}%
,v_{1},...,u_{h},v_{h}),$ then
\begin{equation}%
\begin{array}
[c]{rcl}%
c&=&u_{1}e_{1}\\
v_{1}e_{1}&=&u_{2}e_{2}\\
&\vdots&\\
v_{h}e_{h}&=&c^{\prime}%
\end{array}
\tag{iv}%
\end{equation}
for certain $e_{1},...,e_{h}\in C.$
Equalities (iv) and the left hand side of (ii)
 together constitute a tossing
over $\left[  x\right]  S\cup\left[  x^{\prime}\right]
 S$ and $C$ connecting $(\left[  x\right]  ,c)$
and $(\left[  x^{\prime}\right]  ,c^{\prime}),$ showing
that $C$ is indeed flat. The proof is now complete.
\end{proof}

\section{Axiomatisability of $\mathcal{SF}$}\label{sfsection}

The earliest axiomatisability result, and certainly the
most straightforward, in the sequence of those described in
this paper, is the characterisation of those monoids $S$ such that
$\mathcal{SF}$ is an axiomatisable class.
 The results described in this
section
 appear (in amalgamated form) in
\cite{gould1}. The reader should note that in \cite{gould1}, strongly flat acts are
referred to as flat acts.

For any elements $s,t$ of a monoid $S$, we define right annihilators
$R(s,t)$ and $r(s,t)$ as follows:
\[R(s,t)=\{ u,v)\in S\times S\mid  su=tv\},\]
and
\[r(s,t)=\{ u\in S\mid su=tu\}.\]
Where non-empty, it is clear that $R(s,t)$ and $r(s,t)$ are,
respectively, an S-subact of the right $S$-act $S\times S$ and a
right ideal of $S$.

\begin{proposition}\label{e}
The following conditions are equivalent for a monoid $S$:

(1) the class of left $S$-acts satisfying condition (E) is axiomatisable;

(2) the class of left $S$-acts satisfying condition (E)
is closed under ultraproducts;

(3) every ultrapower of $S$ as a left
$S$-act satisfies condition (E);

(4) for any $s,t\in S$,
$r(s,t)=\emptyset$ or is finitely generated as a
right ideal of $S$.
\end{proposition}
\begin{proof} That (1) implies (2) is immediate from
Theorem~\ref{ultra}; clearly (3) follows from (2) since  $S$
is easily seen to satisfy (E).

Suppose now that every ultrapower of $S$ satisfies condition
(E). Let $s,t\in
S$ and suppose that $r(s,t)\neq\emptyset$ and is {\em not} finitely
generated as a right ideal.

Let $\{ u_{\beta}\mid \beta<\gamma\}$ be a generating set of
$r(s,t)$ of minimum cardinality $\gamma$; we identify the cardinal
$\gamma$ with its initial ordinal; since $\gamma$ is infinite it
must therefore be a limit ordinal.
 Let $\underline{u}\in
\prod_{\beta<\gamma}S_{\beta}$, where each $S_{\beta}$ is a copy of $S$,
be such that $\underline{u}(\beta)=u_{\beta}$. We may suppose that for any
$\beta<\gamma$, $u_{\beta}\notin \bigcup_{\alpha<\beta}u_{\alpha}S$.

From Proposition~\ref{filter1}
 we can choose a uniform ultrafilter $\Phi$
on $\gamma$. Put
$U=(\prod_{\beta<\gamma}S_{\beta})/\Phi$ so that by our assumption
(3),
$U$ satisfies condition (E).

Since $su_{\beta}=tu_{\beta}$ for all $\beta<\gamma$, clearly
$s\underline{u}=t\underline{u}$ and so $s\, \underline{u}\, /\Phi
=t\, \underline{u}\, /\Phi$. Now $U$
has (E), so that  there exist $s'\in S$
and $\underline{v}\, /\Phi \in U$ such that
$ss'=ts'$ and $\underline{u}\, /\Phi =s'\, \underline{v}\, /\Phi$.

From $ss'=ts'$ we have that $s'\in r(s,t)$, so that $s'=u_{\beta}w$
for some $\beta<\gamma$ and $w\in S$. Let $T=\{ \alpha<\gamma\mid
u_{\alpha}=s'v_{\alpha}\}$; from the uniformity of $\Phi$, we can
pick $\sigma\in T$ with $\sigma> \beta$. Then
\[u_{\sigma}=s'v_{\sigma}=u_{\beta}wv_{\sigma}\in u_{\beta}S,\]
a contradiction. We deduce that $r(s,t)$ is finitely generated.

Finally we assume that (3) holds and find a set of axioms for
the class of left $S$-acts satisfying (E).

For any element $\rho$ of $S\times S$ with
 $r(\rho)\neq\emptyset$ we choose and fix a set of
generators
\[w_{\rho 1},\hdots ,w_{\rho\, m(\rho)}\]
of $r(\rho)$. For $\rho=(s,t)$ we define a sentence
$\xi_{\rho}$ of $L_S$ as follows: if
$r(\rho)=\emptyset$ then
\[\xi_{\rho}\leftrightharpoons (\forall x)(sx\neq tx)\]
and
on the other hand, if $r(\rho)\neq\emptyset$ we put
\[\xi_{\rho}\leftrightharpoons (\forall x)\bigg(
sx=tx\rightarrow (\exists z)\bigg(
\bigvee^{m(\rho)}_{i=1}x=w_{\rho\, i}z\bigg)\bigg).\]
We claim that
\[\Sigma_{E}=\{\xi_{\rho}\mid \rho\in S\times S\}\]
axiomatises the class of left $S$-acts satisfying condition (E).

Suppose first that the left $S$-act $A$ satisfies (E),  and let
$\rho= (s,t)\in S\times S$. If $r(\rho)=\emptyset$ and $sa=ta$ for
some $a\in S$, then since $A$ satisfies (E) we have an element
$s'\in S$ such that $ss'=ts'$, a contradiction. Thus $A\vDash
\xi_{\rho}$. On the other hand, if $r(\rho)\neq \emptyset$ and
$sa=ta$ for some $a\in S$, then again we have that $ss'=ts'$ for
some $s'\in S$, and $a=s'b$ for some $b\in A$. Now $s'\in r(\rho)$
so that $s'=w_{\rho\, i}v$ for some $i\in \{ 1,\hdots ,m(\rho)\}$
and $v\in S$. Consequently, $a=w_{\rho\, i}c$ for $c=vb\in A$. Thus
$A\vDash \xi_{\rho}$ in this case also. Thus $A$ is a model of
$\Sigma_{{E}}$.

Finally, suppose that $A\vDash \Sigma_{{E}}$ and $sa=ta$ for some
$s,t\in S$ and $a\in A$. Put $\rho=(s,t)$; since $A\vDash
\xi_{\rho}$ we are forced to have $r(\rho)\neq\emptyset$ and
$a=w_{\rho\, i}b$ for some $i\in \{ 1,\hdots ,m(\rho)\}$. By very
choice of $w_{\rho\, i}$ we have that $sw_{\rho\, i}=tw_{\rho\, i}$.
 Hence $A$
satisfies condition (E) as required.
\end{proof}

Similarly, and argued in full in \cite{gould1}, we have the corresponding
result for condition (P).

\begin{proposition}\label{p}
The following conditions are equivalent for a monoid $S$:

(1) the class of left $S$-acts satisfying condition (P) is axiomatisable;

(2) the class of left $S$-acts satisfying condition (P)
is closed under ultraproducts;

(3) every ultrapower of $S$ as a left $S$-act
 satisfies condition (P);

(4) for any $s,t\in S$,
$R(s,t)=\emptyset$ or is finitely generated as an S-subact
of the right $S$-act $S\times S$.
\end{proposition}

We may put together Propositions~\ref{e} and ~\ref{p} to obtain the
following result for $\mathcal{SF}$, taken from \cite{gould1}.

\begin{theorem}\label{sf}
The following conditions are equivalent for a monoid $S$:

(1) $\mathcal{SF}$ is axiomatisable;

(2) $\mathcal{SF}$
is closed under ultraproducts;

(3) every ultrapower of $S$ as a left $S$-act
is strongly flat;

(4) for any $s,t\in S$,
$r(s,t)=\emptyset$ or is a finitely generated right ideal of $S$, and
$R(s,t)=\emptyset$ or is finitely generated as an S-subact
of the right $S$-act $S\times S$.
\end{theorem}

\section{Axiomatisability of $\mathcal{P}$}\label{projsection}

Those monoids for which $\mathcal{P}$ is
axiomatisable were determined  by the fourth author in
\cite{stepanova}, using preliminary results of the first author from
\cite{gould1}. In fact, $\mathcal{P}$ is axiomatisable if and only if
$\mathcal{SF}$ is axiomatisable and $\mathcal{P}=\mathcal{SF}$. Monoids
for which $\mathcal{P}=\mathcal{SF}$ are called left perfect.
Since left perfect monoids
 figure largely in this and subsequent sections, we devote
some time to them here, developing on the way some new properties of
such monoids.

A left $S$-act $B$ is called a {\em cover} of a left $S$-act $A$ if
there
exists an $S$-epimorphism $\theta:B\rightarrow A$ such that the
restriction of $\theta$ to any proper $S$-subact of $B$ is not an
epimorphism to $A$. If $B$ is in addition projective, then $B$ is a
{\em projective cover} for $A$. A monoid $S$ is {\em left perfect}
if every left $S$-act has a projective cover.

We now give a number of finitary conditions used in determining left
perfect monoids, and in subsequent arguments.

(A) Every left $S$-act satisfies the ascending chain condition for
cyclic
$S$-subacts.

(D) Every right unitary submonoid of $S$ has a minimal left ideal
generated by an idempotent.

($M_R$)/($M_L$) The monoid $S$ satisfies the descending
chain condition for principal right/left ideals.

($M^R$)/($M^L$) The monoid $S$ satisfies the ascending
chain condition for principal right/left ideals.

\begin{theorem}\label{perfect}\cite{fountain,isbell, kilp} The following
conditions
are equivalent for a monoid $S$:

(1) $S$ is left perfect;

(2) $S$ satisfies Conditions (A) and (D);

(3) $S$ satisfies Conditions (A) and ($M_R$);

(4) $\mathcal{SF}=\mathcal{P}$.
\end{theorem}

\begin{proposition}\label{perfectagain} Let $S$ be a left perfect monoid. Then

(1) $S$ is group bound;

(2)
if $Sb\, (bS)$ is a minimal left (right) ideal of $S$, then $bS$
($Sb$) is a minimal right (left) ideal of $S$;

(3) if $Sb_1\subseteq Sb_0$ and $Sb_1\cong Sb_0$, then $Sb_0=Sb_1$;

(4) any minimal left (right) ideal of $S$ is generated by an
idempotent.
\end{proposition}
\begin{proof} (1) Let $S$ be  a left perfect monoid. From
Theorem~\ref{perfect}
$S$ satisfies ($M_R$), so that for any $a\in S$,
$a^mS=a^{m+1}S$ for some $m\in \mathbb{N}$. On the other hand, consider
the descending chain
\[Sa\supseteq Sa^2\supseteq Sa^3\supseteq\hdots \]
of principal left ideals. Let $\Phi$ be a uniform ultrafilter on
$\mathbb{N}$ and consider the ultrapower $U=S^{\mathbb{N}}/\Phi$.
For each $n\in \mathbb{N}$ let
\[\underline{u_n}=(1,1,\hdots ,a,a^2,a^3,\hdots)/\Phi \]
where the first $a$ occurs in the $n$'th place. Clearly
$\underline{u_n}=a\underline{u_{n+1}}$ for any $n\in\mathbb{N}$, so that
\[S\underline{u_1}\subseteq S\underline{u_2}\subseteq\hdots\]
By Theorem~\ref{perfect}, $U$ has the ascending chain condition on
cyclic
$S$-subacts, so that $S\underline{u_h}=
S\underline{u_{h+1}}$ for some $h$. Consequently,
$s\underline{u_h}=\underline{u_{h+1}}$ for some $s\in S$; since
$\Phi$ is uniform, we deduce that for some
$i\geq h+1$, $sa^{i-h+1}=a^{i-h}$. Putting $k=i-h$ we deduce that
$a^k\el\, a^{k+1}$. Now take $n$ to be the bigger of $m$ and $k$;
clearly $a^n\,\eh\, a^{n+1}\, \eh a^{2n}$, whence by Theorem~\ref{green},
 $a^n$ lies in a subgroup of $S$.

(2) Suppose now that $Sb$ is a minimal left ideal of $S$; since $S$ has
($M_R$) we can choose $c\in S$ with $cS\subseteq bS$ and $cS$ minimal.
Notice that $Scb=Sb$ and so $cb\,\el\, b$; consequently
$cb\,\els\, b$. If $d\in bS$, then
as $c^2S=cS=cdS$ we have $cd=ccd'$ for some $d'\in S$. Now,
$d=bx, cd'=by$ for some $x,y\in S$, and so
$cbx=cby$, giving that
\[d=bx=by=cd',\]
that is, $d\in cS$. Hence $bS=cS$ is minimal.

To prove (3), let us assume that $Sb_1\subseteq Sb_0$ and
$Sb_1\cong Sb_0$. Let
$\phi:Sb_0\rightarrow Sb_1$ be an $S$-isomorphism. Then
$\phi(b_0)=sb_1=b_2$ for some $s\in S$.
Then $Sb_2\subseteq Sb_1\subseteq Sb_0$ and
$b_2=tb_0$ for some $t\in S$. Since $\phi$
is an isomorphism  we have that $tb_0\,\ars\, b_0$ and as $\ars$ is a left congruence and
$S$ is group bound, $b_0\,\ars\, t^nb_0$ for some
$n\in\mathbb{N}$ such that $t^n$ lies in a subgroup. Let $s$ be the
inverse
of $t$ in this subgroup. Then $st^nt^n=t^n$, so that
 \[b_0=st^nb_0=st^{n-1}tb_0=st^{n-1}b_2\]
whence $Sb_2=Sb_1=Sb_0$ as required.

We now prove the second part of (2). Suppose $b\in S$ and $bS$ is a
minimal right ideal. Let $Sc\subseteq Sb$. In view of the minimality of
$bS$ we have that $bc\,\ar\, b$ and so $bc\,\ars\, b$. By
Lemma~\ref{pp},
$Sbc\cong Sb$. But $Sbc\subseteq Sc\subseteq Sb$. Now (3) gives that
$Sbc=Sc=Sb$ as required.

To see that (4) holds, note that if $Sb$ is a minimal left ideal,
then $Sb=Sb^n$ for all $n\in \mathbb{N}$; since $S$ is group bound,
$b^n\,\eh\, e$ for some $n\in\mathbb{N}$ and some $e\in E$. Hence
$Sb^n=Se$; dually for principal right ideals.
\end{proof}

From Propositions~\ref{groupboundlocal} and \ref{perfectagain} the
following is immediate.

\begin{corollary}\label{local} Let $S$ be a left perfect monoid.
Then $S$ is local
and $\mathcal{D}=\mathcal{J}$.
\end{corollary}

Before stating the main result of this section, we require a
preliminary lemma, due to the fourth author, that has significant
consequences. We recall from Section~\ref{ma} say that a monoid
satisfies (CFRS) if
\[\forall s\in S\; \exists n_s\in\mathbb{N} \; \forall t\in S
\;|\{ x\in S|\, sx=t\}|\leq n_s.\]

\begin{lemma}\cite{stepanova}
\label{crfs}  Let $S$ be a monoid such that every ultrapower of $S$ as a
left
$S$-act is projective.  Then $S$ satisfies (CFRS).
\end{lemma}

\begin{proof}
Suppose to the contrary that $t\in S$ exists for which the condition
is \emph{not} true.  That is, for each $n\in \mathbb{N} $, there
exists an element $a_n\in S$ such that $\left| \{ x\in S\mid
tx=a_n\}\right| >n.$ Let $\Phi$ be a uniform ultrafilter on
$\mathbb{N}$ and for $n\in \mathbb{N}$ choose $b_{n,1},\hdots
,b_{n,n}\in S$ such that $tb_{n, i}=a_n,\, 1\leq i\leq n$. Put
\[\underline{c_n}=(1,\hdots ,1,b_{n,n},b_{n+1, n},b_{n+2, n},\hdots
)\] with $b_{n,n}$ occuring in the $n$'th place. Now,
$\underline{c_i}\, /\Phi\neq \underline{c_j}\, /\Phi$ for any $i\neq
j$, but $t\underline{c_i}\, /\Phi=\underline{a}\, /\Phi$ where
$\underline{a}=(a_1,a_2,\hdots )$. Since $U=S^\mathbb{N}/\Phi$ is
projective and $\underline{a}\, /\Phi \in S\underline{c}\,/\Phi\cong
Se$ for some $e\in E$, we deduce that there exists $d\in S$ such
that $A=|\{ x\in S\mid tx=d\}|$ is infinite.

Now choose a cardinal $\alpha> |S|$. Combining
Propositions~\ref{filter2}
and \ref{filter3}, we can choose an ultrafilter $\Theta$ over
$\alpha$ such that $|A^{\alpha}/\Theta|=|A|^{\alpha}>|S|$ (any set may
be regarded as an $S$-act over a trivial monoid). Put
$V=S^{\alpha}/\Theta$ and let $\underline{d}\in S^{\alpha}$ be such that
$\underline{d}(i)=d$ for all $i\in\alpha$. If
$\underline{x}\in A^{\alpha}$ then clearly $t\,\underline{x}\, /\Theta
=\underline{d}\, /\Theta$. But $V$ is projective by assumption, so that
$\underline{d}\,/\Theta\in S \underline{g}\, /\Theta\cong Sf$ for some
$f\in E$. Consequently, there is an element $u\in S$ such that the
equation
$tx=u$ has more than $|S|$ solutions in $S$, which is clearly
nonsense. Hence $S$ has
(CFRS).
\end{proof}

\begin{corollary}\label{crfsproj} Let $S$ be such that any ultrapower of $S$ as a left
$S$-set is projective. Then for all $s\in S$ there exists $n_s\in
\mathbb{N}$ such that for any $P\in\mathcal{P}$ and $t\in P$,
\[|\{ x\in S|\, sx=t\}|\leq n_s.\]
\end{corollary}

We now set out to prove the main result of this section, due to the
first and fourth  authors.

\begin{theorem}\label{projax}  \cite{gould1,stepanova,BG} The following conditions are
equivalent
for a monoid $S$:

(1) every ultrapower of the left $S$-act $S$ is projective;

(2)  $\mathcal{SF}$ is axiomatisable
and $S$ is left perfect.

(3) $\mathcal{P}$ is axiomatisable.
\end{theorem}
\begin{proof} If $\mathcal{SF}$ is axiomatisable and $S$ is left
perfect, then by Theorem~\ref{perfect}, $\mathcal{P}=
\mathcal{SF}$ is an axiomatisable class. Clearly if $\mathcal{P}$
is axiomatisable, then Theorem~\ref{ultra} gives
that every ultrapower of $S$ is projective.

Suppose that every ultrapower of $S$ as a left $S$-act is
projective; by Theorem~\ref{sf}, certainly $\mathcal{SF}$ is an
axiomatisable class. We proceed via a series of subsidiary lemmas.

\begin{lemma}\label{M_R}\cite{gould1} Let $S$ be such that every ultrapower of the
left
$S$-act $S$
is projective. Then $S$ has ($M_R$).
\end{lemma}
\begin{proof} Let $a_1S\supseteq b_2S\supseteq
b_3S\supseteq\hdots$ be a decreasing sequence of principal right ideals
of $S$, so that for  $i\geq 2$ we have $b_i=b_{i-1}a_i$ for some $a_i\in S$,
putting  $b_1=a_1$. Thus $b_2=a_1a_2, b_3=b_2a_3=a_1a_2a_3,\hdots$.

Let $\Phi$ be a uniform ultrafilter over $\mathbb{N}$ and put
${U}=S^{\mathbb{N}}/\Phi$; by assumption, ${U}$ is
projective.

Define elements $\underline{u_i}\in S^{\mathbb{N}}$, $i\in\mathbb{N}$, by
\[\underline{u_i}=(1,1,\hdots ,1,a_i,a_ia_{i+1},
a_ia_{i+1}a_{i+2},\hdots),\]
where the entry $a_i$ is the $i$'th coordinate. Then, for any
$i,j\in\mathbb{N}$ with $i<j$ we have
\[\underline{u_i}\, /\Phi=a_ia_{i+1}\hdots a_{j-1}\underline{u_j}\, /\Phi.\]
Since ${U}$ is projective,  Proposition~\ref{proj}
and Corollary~\ref{cyclicproj} give that
\[S\underline{u_1}\, /\Phi\subseteq S\underline{u_2}\, /\Phi
\subseteq \hdots S\underline{c}\, /\Phi\]
where $\underline{c}\, /\Phi$ is left $e$-cancellable for some $e\in E$.
Put $\underline{c}=(c_1,c_2,\hdots )$, and let
$d_i\in S$ be such that $\underline{u_i}\, /\Phi
=d_i\underline{c}\, /\Phi$ for each $i\in\mathbb{N}$. For any
$i<j$ we have that
\[d_i\underline{c}\, /\Phi= a_i\hdots a_{j-1}d_j\underline{c}\, /\Phi\]
whence from the left $e$-cancellability of $\underline{c}\, /\Phi$,
\[d_ie=a_i\hdots a_{j-1}d_je.\]
Choose
 $i\in\mathbb{N}$
such that $a_1\hdots a_i=d_1c_i$, and $ec_i=c_i$. Then for any $j>i$,
\[a_1\hdots a_iS=d_1ec_iS=a_1\hdots a_{j}d_{j+1}ec_iS
\subseteq a_1\hdots a_jS\subseteq a_1\hdots a_iS,\]
whence
\[b_iS=b_{i+1}S=\hdots\]
as required.
\end{proof}

\begin{lemma}\label{M^L}\cite{stepanova} Let $S$ be such that every ultrapower of the
left
$S$-act $S$
is projective. Then $S$ has ($M^L$).
\end{lemma}

\begin{proof}
Consider an ascending chain
\begin{equation*}
Sa_{1}\subseteq Sa_{2}\subseteq \cdots
\end{equation*}
of principal left ideals of $S.$ Let $u_{2},u_{3},\hdots \in S$ be such that $%
a_{i}=u_{i+1}a_{i+1},$ from which it follows that for any $i,$ $j$ with $i<j,
$%
\begin{equation*}
a_{i}=u_{i+1}u_{i+2}\cdots u_{j}a_{j}.
\end{equation*}
Consider an ultrapower ${U}=S^{\mathbb{N} }/\Phi $ where $\Phi $ is a
uniform ultrafilter on $\mathbb{N} .$ By assumption, ${U}$ is
projective as a left $S$-act. For each $i\geq 2$ define
\begin{equation*}
\underline{v_{i}}\, /\Phi=( 1,1,\hdots ,1,u_{i},u_{i}u_{i+1},
u_{i}u_{i+1}u_{i+2},\hdots
) /{\Phi },
\end{equation*}
the entry $u_{i}$ occurring in the $i$th position. Observe that
\begin{equation*}
\underline{v_{i}}\, /\Phi=u_{i}\underline{v_{i+1}}\, /\Phi,
\end{equation*}
for each $i$, and as in Lemma~\ref{M_R} the projectivity of
${U}$
ensures that
 there exist $\underline{f}\, /\Phi=
(f_{1},f_{2},...)/{\Phi }$  and $s_1,
s_{2},\hdots \in S$ such that
\begin{equation*}
\underline{v_{k}}\, /\Phi=s_{k}\underline{f}\, /\Phi
\end{equation*}
for each $k\in\mathbb{N}$. For each natural number $k$ it follows that the set
\begin{equation*}
E_{k}=\left\{ j\in \mathbb{N} \mid j>k,\,u_{k}u_{k+1}\cdots
u_{j}=s_{k}f_{j}\right\}
\end{equation*}
belongs to $\Phi .$ For each $i$ we put
\begin{equation*}
\underline{g_{i}}\, /\Phi=f_{i}\underline{v_{i+1}}\, /\Phi=(
f_{i},...,f_{i},f_{i}u_{i+1},f_{i}u_{i+1}u_{i+2},...)\, /\Phi.
\end{equation*}

Suppose that for each $i\in \mathbb{N} ,$ $\left\{ j\in \mathbb{N} \mid
\underline{%
g_{i}}\, /\Phi=\underline{g_{j}}\, /\Phi
\right\} \notin \Phi .$ In this case, for each $i\in
\mathbb{N} $ put $T_{i}=\left\{ j\in \mathbb{N} \mid \underline{g_{i}}\, /\Phi
\neq \underline{%
g_{j}}\, /\Phi\right\} ,$
 which by assumption belongs to $\Phi .$
 Now
choose $j_{1},j_{2},...\in \mathbb{N} $ as follows:
\begin{eqnarray*}
j_{1} &\in &E_2\\
j_{2} &\in &E_2\cap T_{j_{1}},\,j_{2}>j_{1} \\
j_{3} &\in &E_2\cap T_{j_{1}}\cap T_{j_{2}},\,j_{3}>j_{2}>j_{1} \\
&&\vdots
\end{eqnarray*}
For any $k,$
\begin{eqnarray*}
\underline{v_{2}}\, /\Phi &
=&u_{2}u_{3}\cdots u_{j_{k}}\underline{v_{j_{k}+1}}\, /\Phi \\
&=&s_{2}f_{j_{k}}\underline{v_{j_{k}+1}}\, /\Phi \\
&=&s_{2}\underline{g_{j_{k}}}\, /\Phi.
\end{eqnarray*}
By definition of the sets $T_i$,
 all of the elements $\underline{g_{j_{k}}}\, /\Phi$ are
distinct. This contradicts Corollary~\ref{crfsproj}.

In view of the above, there exist $i_{0}\in \mathbb{N} $ such that
$D=\left\{ j\in \mathbb{N} \mid \underline{g_{j}} \,
/\Phi=\underline{g_{i_{0}}}\, /\Phi\right\} $ belongs to $\Phi .$
For any $k\in D$ with $k>i_{0},$ pick $j\in D\cap E_{k}. $ Then we
have
\begin{equation*}
S_{j}=\left\{ m\in \mathbb{N} \mid m>j,\,f_{j}u_{j+1}\cdots
u_{m}=f_{i_{0}}u_{i_{0}+1}\cdots u_{m}\right\} \in \Phi .
\end{equation*}
Now take any $m\in S_{j}$ and calculate
\begin{eqnarray*}
a_{k-1} &=&u_{k}\cdots u_{j}u_{j+1}\cdots u_{m}a_{m} \\
&=&s_{k}f_{j}u_{j+1}\cdots u_{m}a_{m} \\
&=&s_{k}f_{i_{0}}u_{i_{0}+1}\cdots u_{m}a_{m} \\
&=&s_{k}f_{i_{0}}a_{i_{0}}.
\end{eqnarray*}
In summary we have shown in this case that, for any $k\in D$ with $k>i_{0},$
the left ideals $Sa_{k-1}$ and $Sa_{i_{0}}$ are equal. If now $l>i_{0}$ is
arbitrary, take any $k\in D$ with $k>l$ and note that $Sa_{i_{0}}\subseteq
Sa_{l}\subseteq Sa_{k-1}=Sa_{i_{0}},$ and so the chain terminates, as
required.

\end{proof}

In order to  complete the proof of the implication $(1)\Rightarrow (2)$
in
Theorem~\ref{projax} we need one further lemma.
 To this end, the following alternative
characterisation of Condition (A) will be useful.

\begin{lemma}\label{a}\cite{isbell} A monoid $S$ satisfies (A) if and
only if for any
elements $a_1,a_2,\hdots $ of $S$, there exists $n\in\mathbb{N}$ such
that
for any $i\in\mathbb{N}$, $i\geq n$, there exists $j_i\in\mathbb{N},
j_i\geq i+1$, with
\[Sa_ia_{i+1}\hdots a_{j_i}=Sa_{i+1}\hdots a_{j_i}.\]

\end{lemma}

\begin{lemma}\label{lastpiece}\cite{gould1} Let $S$ be such that every ultrapower of the
left
$S$-act $S$
is projective. Then $S$ satisfies Condition (A).
\end{lemma}
\begin{proof} Let $a_i\in S, i\in\mathbb{N}$ and define $\Phi$,
${U},
\underline{u_i}\, /\Phi, d_i (i\in\mathbb{N})$
and $\underline{c}\, /\Phi$ as in the proof of
Lemma~\ref{M_R}. For any $i$ we have that
$d_ie=a_id_{i+1}e$ so that $Sd_1e\subseteq Sd_2e\subseteq \hdots$. By
Lemma~\ref{M^L}
we know that $Sd_ne=Sd_{n+1}e=\hdots$ for some $n\in\mathbb{N}$ and
since $e\underline{c}\, /\Phi=\underline{c}\, /\Phi$ it
follows that $S\underline{u_n}\, /\Phi= S\underline{u_{n+1}}\,
/\Phi
=\hdots$. Now let $i\geq n$, so that $\underline{u_{i+1}}\, /\Phi
=s\underline{u_i}\, /\Phi$ for some $s\in S$. Since $\Phi$ is
uniform
there exists $j_i\geq i+1$ such that
$a_{i+1}\hdots a_{j_i}=sa_ia_{i+1}\hdots a_{j_i}$ and so
\[Sa_{i+1}\hdots a_{j_i}=Sa_ia_{i+1}\hdots a_{j_i},\]
so that by Lemma~\ref{a}, $S$ satisfies Condition (A).

\end{proof}

We now proceed with the proof of Theorem~\ref{projax}. If every
ultrapower of $S$ is projective, then from Lemmas~\ref{M_R}
and~\ref{lastpiece} we know that $S$ satisfies ($M_R$) and Condition
(A). By Theorem~\ref{perfect}, $S$ is left perfect.
\end{proof}

\section{Axiomatisability of $\mathcal{F}r$}\label{freesection}

The question of axiomatisability of $\mathcal{F}r$ was solved in
some special cases by the fourth author in \cite{stepanova}, and
most recently by the first author as below.

For convenience, we introduce some new terminology. Let $e\in E$ and
$a\in S$. We say that $a=xy$ is an {\em $e$-good factorisation of
$a$ through $x$} if $y\neq wz$ for any $w,z$ with $e=xw$ and $w\,\el\, e$.

\begin{theorem}\label{freesax}  The following conditions are
equivalent
for a monoid $S$:

(1) every ultrapower of the left $S$-act $S$ is free;

(2)  $\mathcal{P}$ is axiomatisable and $S$ satisfies
(*):
for all $e\in E
\setminus\{ 1\}$, there exists a finite set
$f\in S$ such that any
$a\in S$ has an $e$-good factorisation through $w$, for
some $w\in f$.

(3) $\mathcal{F}r$ is axiomatisable.
\end{theorem}
\begin{proof} If  $\mathcal{F}r$ is axiomatisable,
then certainly (1) holds. On the other hand, if (1) holds,
then by Theorem~\ref{projax}, $\mathcal{P}$ is axiomatisable and $S$ is
left
perfect. Note that by Corollary~\ref{local}, $S$ is local.
We show that (*) holds.

Let $e\in E$ with $e\neq 1$. For any $a\in S$, $a=a\cdot 1$; if $e=av$
with $v\,\el\, e$, then as $S$ is local, $1\neq vc$ for any $c$.

We proceed by contradiction. Let $J$ denote the set of finite subsets of
$S$.
Suppose that for any $f\in J$ there exists an element $w_f\in S$
such that
$w_f$ does not have an $e$-good factorisation
through $w$, for any $w\in f$. Clearly $S$ and $J$ must be infinite.

For each $w\in S$ let $J_w=\{ f\in J\mid w\in f\}$; since $\{
w_1,\hdots ,w_n\}\in J_{w_1}\cap\hdots \cap J_{w_n}$, there exists
an ultrafilter $\Phi$ over $J$ such that $J_w\in \Phi$ for all $w\in
S$.

Consider ${U}=S^J/\Phi$; by assumption (1), ${U}$ is
free.
Let $\underline{x}\in S^J$ be such that $\underline{x}(f)=w_f$. Since
${U}$ is free, Theorem~\ref{free}
and Corollary~\ref{freecyclic} give that
$\underline{x}\, /\Phi=w\underline{d}\, /\Phi$ for some $w$,
where $\underline{d}\, /\Phi$ is right $1$-cancellable. Suppose
that $\underline{d}(f)=d_f$, for any $f\in J$.

We claim that
\[D=\{ f\in J\mid wd_f \mbox{ is an }e\mbox{-good factorisation through } w\}
\in\Phi.\]
Suppose to the contrary. Then
\[D'=\{ f\in J\mid d_f=vz\mbox{ for some }v,z\mbox{ with }e=wv\mbox{ and
}v\,\el\,e\}\in \Phi.\]
By Lemma~\ref{crfs}, there are only finitely many $v_1,\hdots ,v_n$
such that $e=wv_i$ and $v_i\, \el\, e$. For $1\leq i\leq n$ let
\[D_i=\{ f\in J\mid d_f=v_iz\mbox{ for some }z\},\] so that
$D'=D_1\cup\hdots \cup D_n$. Consequently, $D_i\in \Phi$ for some
$i\in\{ 1,\hdots ,n\}$. We know that $v_i\,\el\, e$, so that $v_i$
is regular and $v_i\,\ar\, g$ for some $g\in S$; as $S$ is local,
$g\neq 1$. But $gv_i=v_i$ so that $gd_f=d_f$ for all $f\in
D_i\in\Phi$. Hence $g\underline{d}\, /\Phi=\underline{d}\, /\Phi$,
so that as $\underline{d}\, /\Phi$ is right $1$-cancellable, $g=1$,
a contradiction. We conclude that $D\in\Phi$.

Let $T=\{ f\in J:w_f=wd_f\}$, so that $T\in\Phi$; now pick
$f\in D\, \cap T\, \cap J_w$. We have that $w\in f$, and as $f\in T$,
$w_f=wd_f$; moreover, as $f\in D$, this is an $e$-good
factorisation of $w_f$ through $w$. This contradicts the choice
of $w_f$. We deduce that (*) holds.

Finally, we suppose that  $\mathcal{P}$ is axiomatisable,
  and $S$ satisfies
(*). Let $\Sigma_{\mathcal{P}}=\Sigma_{\mathcal{SF}}$ be a set of sentences axiomatising
$\mathcal{P}$. Let $e\in E,e\neq 1$. Choose
a finite set  $f=\{ u_1,\hdots ,u_n\}$ guaranteed by (*), such
that every $a\in S$ has an $e$-good factorisation through $u_i$, for
some $i\in \{ 1,\hdots ,n\}$.  Since $\mathcal{P}$ is axiomatisable,
Lemma~\ref{crfs} tells us that for each $i\in \{ 1,\hdots ,n\}$ there exist finitely many
$v_{i1},\hdots ,v_{im_i}\in L_e$, $m_i\geq 0$, such that $e=u_iv_{ij},
1\leq j\leq m_i$. Let
\[\varphi_{e,i}\leftrightharpoons
(\exists b)
(a=u_ib\wedge(\bigwedge_{1\leq j\leq m_i} b\neq
v_{ij}a)).\]
We now define $\phi_e$ as
\[(\forall a)\bigvee_{1\leq i\leq n}\varphi_{e,i}.\]
Put
\[\Sigma_{\mathcal{F}r}=\Sigma_{\mathcal{P}}\cup\{ \varphi_e\mid e\in
E\setminus \{ 1\}\}.\]
We claim that $\Sigma_{\mathcal{F}r}$ axiomatises $\mathcal{F}r$.

Let $F$ be a free $S$-set; certainly $F\vDash \Sigma_{\mathcal{P}}$. Say
that
$F$ is free on $X$, let $e\in E, e\neq 1$
 and let $a\in F$. Then $a=sx$ for some $x\in X$. By choice of
$u_1,\hdots ,u_n$, we can write $s=u_it$ for some $t\in S$ with
$t\neq vw$ for any $w\in S$ and $v\in L_e$ such that $e=u_iv$. Put $b
=tx$; clearly then $F\vDash \varphi_e$.

Conversely, let $A$ be an $S$-set and suppose that $A\models
\Sigma_{\mathcal{F}r}$. Since $A$ is therefore projective, we know that
$A$ is a coproduct of maximal indecomposable $S$-subsets of the form
$Sa$, where there exists an $e\in E$ such that
$a$ is right $e$-cancellable;
notice that $ea=a$. Suppose that $e\neq 1$. Since $A\models \varphi_e$ we have that
$a=u_ib$ for some $b$ such that $b\neq va$ for
any $v\in L_e$ with $e=u_iv$. But $b=wa$ say, giving that
$a=u_iwa$ and so $e=u_iwe$. Clearly $e\,\el\, we$,
and  $b=wa=wea$,  a contradiction. Thus
$e=1$ and we deduce that $A$ is free. Consequently,
$\Sigma_{\mathcal{F}r}$
axiomatises $\mathcal{F}r$ as required.

\end{proof}

For some restricted classes of monoids, we can simplify the condition
given in Theorem~\ref{freesax}. We say that the group of units $H_1$
of a monoid $S$ has {\em finite right index} if there exists
$u_1,\hdots ,u_n\in S$ such that $S=u_1H_1\cup\hdots u_nH_1$. Note that
if in addition $S$ is local, then for any $e\in E, e\neq 1$,
 any $a\in S$ has an $e$-good
factorisation through $u_i$, for some $i\in\{ 1,\hdots ,n\}$.

\begin{proposition}\label{finiteideals} Let $S$ be a monoid
such that
\[S\setminus R_1= s_1S\cup\hdots \cup s_mS\]
for some $s_1,\hdots ,s_m\in S$. Then $\mathcal{F}r$ is an
axiomatisable class if and only if $\mathcal{P}$ is axiomatisable and
$H_1$ has finite right index in $S$.

\end{proposition}
\begin{proof} Suppose that $\mathcal{P}$ is axiomatisable and $H_1$ has
finite right index in $S$. By Theorem~\ref{projax}
Corollary~\ref{local}, $S$ is
local, so that by the comments above, condition
(*) of Theorem~\ref{freesax} holds, and so $\mathcal{F}r$ is an axiomatisable class.

Conversely, if $\mathcal{F}r$ is axiomatisable, it remains only to show
that $H_1$ has finite right index. Suppose for contradiction that there
exists
$a_1,a_2,\hdots $ in $S$ with $a_iU\cap a_jU=\emptyset$ for all $i\neq
j$. Let $\Phi$ be a uniform ultrafilter on $\mathbb{N}$, let
${U}=S^{\mathbb{N}}/\Phi$ and let
$\underline{a}\in S^{\mathbb{N}}$ be given by
$\underline{a}(i)=a_i$. Since ${U}$ is free,
$\underline{a}\, /\Phi= w\underline{d}\, /\Phi$ for some right
$1$-cancellable $\underline{d}\, /\Phi$ generating the
connected component in which $\underline{a}\, /\Phi$ lies.
 Say $\underline{d}(i)=d_i$.

 By Theorem~\ref{freesax},
and Corollary~\ref{local}, we know that $S$ is local so that
$R_1=H_1$. Hence $\mathbb{N}=T_1\cup\hdots T_m\cup T$ where $T_i=\{
i\in\mathbb{N}\mid d_i\in s_iS\}$ and $T=\{ i\in \mathbb{N}\mid
d_i\in H_1\}$. If $T_i\in\Phi$, then $\underline{d}\, /\Phi=
s_i\underline{f}\, /\Phi$ for some $\underline{f}\, /\Phi$; but
$\underline{f}\, /\Phi= v\underline{d}\, /\Phi$ so that as
$\underline{d}\, /\Phi$ is $1$-cancellable, we obtain $1=s_iv$, a
contradiction. Hence $T\in\Phi$. Let $D=\{
i\in\mathbb{N}:a_i=wd_i\}$, and pick distinct $i,j\in D\cap T$. Then
$a_i=wd_i,a_j=wd_j$ and so
\[a_iH_1 =wd_iH_1=wH_1=wd_jH_1=a_jH_1,\]
a contradiction. We deduce that $H_1$ has finite right index in $S$.

\end{proof}

Our final corollary is now straightforward.

\begin{corollary}\label{inverse}\cite{stepanova} Let $S$ be an inverse monoid. Then
$\mathcal{F}r$ is an axiomatisable class if and only if $\mathcal{P}$ is
axiomatisable and $H_1$ has finite right index in $S$.

\end{corollary}
\begin{proof} The converse holds as in
Proposition~\ref{finiteideals}.

Suppose now that $\mathcal{F}r$ is axiomatisable. Since $S$ has
$(M_R)$ we can pick a minimal principal right ideal; as $S$ is
regular, this is generated by $e\in E$. For any $f\in E$ we have
that $eS=efS$, so that $e\,\ar\, ef$. But $S$ is inverse, so that
$E$ is a semilattice, and every $\ar$-class contains a unique
idempotent. Hence $e=ef$ for all $f\in E$; by Lemma~\ref{crfs} we
deduce that $E$ is finite. Since every principal right ideal is
idempotent generated, $S$ has only finitely many principal right
ideals. The result follows by Proposition~\ref{finiteideals}.
\end{proof}

\section{ Completeness, model completeness and categoricity}\label{cmc}

At this point we need to present a little more model theory as
motivation for the remaining sections. We remind the reader that
throughout, $L$ denotes a first order language.

An {\em elementary theory} or simply a {\em theory}
 of a first order language $L$ is a set
of  sentences $T$ of $L$, which is closed under
deduction. We recall from
Section~\ref{ax} that an $L$-structure $\mathbf{ A}$ in which all
sentences of the theory $T$
are true, is called a {\em model} of the theory $T$ and we
write $\mathbf{A}\models T$. A theory $T$ is {\em consistent} if for any
sentence $\varphi\in L$ we do not have both $\varphi$ and
$\neg\varphi\in
T$. A consistent theory
$T$  is {\em complete} if $\varphi\in T$ or
$\neg\varphi\in T$ for any sentence $\varphi$ of the language $L$.
By the {\em Extended Completeness Theorem} (c.f. Theorem 1.3.21 of
\cite{changkeisler}), a theory $T$ is consistent if and only if it has a
 model.
Clearly, for any $L$-structure $\mathbf{A}$, the {\em theory}
Th$(\mathbf{A})$
 {\em of $\mathbf{A}$}, defined by
\[\mbox{Th}(\mathbf A)=\{ \varphi\mid \varphi\mbox{ is a sentence, }\mathbf A\models
\varphi\}\] is a consistent complete theory and for a class of
$L$-structures $\mathcal{K}$,
\[\mbox{Th}(\mathcal{K})=\bigcap_{\mathbf{A}\in\mathcal{K}}\mbox{Th}
(\mathbf A)\]
is consistent.

Structures $\mathbf{A,B}$ for $L$ are {\em elementarily equivalent},
denoted $\mathbf{A}\equiv \mathbf{B}$, if Th$(\mathbf{A})=$
Th$(\mathbf{B})$. One of the basic tenets of model theory tells us that
if $T$ is a theory and $\varphi$ is a sentence such that $\varphi\notin
T$,
then $T\cup\{ \neg\varphi\}$ is consistent, and hence has a model. Thus
we deduce

\begin{lemma}\label{complete} A consistent theory $T$ is complete if and
only if $\mathbf{A}\equiv\mathbf{B}$ for any models $\mathbf{A,B}$ of
$T$.
\end{lemma}

 The previous few chapters have concentrated on
axiomatisable classes of $S$-acts. We remark here that if $\mathcal{K}$
is
a class axiomatised by $T$, then Th$(\mathcal{K})=T$.

 For an $L$-structure $\mathbf{A}$ with
universe $A$ and subset $B$ of $A$ we often consider the augmented
or {\em enriched} language $L_B$, which is obtained from $L$ by
adding a set of constants $\{b'\mid b\in B\}$, where $b_1'\neq b_2'$
for distinct $b_1$ and $b_2$ from $B$. We will write  $\mathbf{A}_B$
for the corresponding enriched $L_B$-structure. So, $\mathbf{A}_B$
is obtained from $\mathbf{A}$ by interpreting the constant $b'$,
where $b\in B$, by $b$. We denote by Th$(\mathbf{A},b)_{b\in B}$ the
set of sentences of $L_B$ true in
 $\mathbf{A}_B$.  Another useful tool is that of the {\em diagram}
of an $L$-structure $\mathbf{A}$, denoted by Diag $\mathbf{A}$,
which is
 the set of atomic and negated atomic formulae of $L_A$ that are
true in  $\mathbf{A}_A$. From Proposition 2.1.8 of
\cite{changkeisler}, an $L$-structure $\mathbf{A}$ embeds into an
$L$-structure $\mathbf{B}$ if and only if $\mathbf{B}$ has an
enriching to the language  $L_A$, such that
$\mathbf{B}_A$ is a model of Diag
$\mathbf{A}$.

We adopt the standard convention of writing $\bar x\in X$ to
indicate that $\bar x=(x_1,\hdots ,x_n)$ for some finite set
$\{ x_1,\hdots,x_n\}\subseteq X$. A
substructure $\mathbf{A}$ of an $L$-structure $\mathbf{B}$ is
said to be
{\em elementary}  (denoted  ${\mathbf A\preccurlyeq \mathbf B}$), if for any formula
$\varphi(\bar x)$ of the language $L$ and any $\bar a\in A$
$${\mathbf A}\models\varphi(\bar a)\Leftrightarrow{\mathbf
B}\models\varphi(\bar a).$$ Note that in this definition the
condition``$\Leftrightarrow$'' can be exchanged to ``$\Leftarrow$''
or ``$\Rightarrow$'' (consider the negation of the formula and bear
in mind that for any
$L$-structure $\mathbf{C}$,
Th$(\mathbf{C},c)_{c\in C}$ is complete). It is easy to see
that if
$\mathbf A$ is a substructure of $\mathbf B$, then
\[\mathbf A\preccurlyeq \mathbf B\mbox{ if and only if }
\mathbf B_A\models\mbox{Th}(A,a)_{a\in A}.\]

We can now state a crucial result, known as the upward and downward
L\"{o}wenheim-Skolem-Tarski theorem.

\begin{theorem} (Corollary 2.1.6 and Theorem 3.1.6) \cite{changkeisler}
\label{updownlst} Let $T$ be a
theory in $L$ with an infinite model. Then for any cardinal
$\varkappa\geq |L|$, $T$ has a model $\mathbf A$ with $|A|=\varkappa$.

If $\mathbf B$ is a model of $T$ with $|B|=\varkappa\geq \alpha\geq |L|$
and
$X\subseteq B$ with $|X|\leq \alpha$, then there is an elementary
substructure $\mathbf C$ of $\mathbf B$ (so that certainly
$\mathbf C\models T$) such that $X\subseteq C$ and $|C|=\alpha$.

\end{theorem}

 A consistent
theory $T$ of the language $L$ is called {\em model complete}
 if $$\mathbf
A\subseteq \mathbf B\Rightarrow \mathbf A\preccurlyeq \mathbf B $$ for any models
$\mathbf{A,B}$
of $T$.

\begin{lemma}\label{modelcompleteimpliescomplete}
(Proposition 3.1.9)  \cite{changkeisler} Let $T$ be model complete
and such that any two models of $T$ are isomorphically embedded in a
third. Then $T$ is complete.

\end{lemma}
\begin{proof} Let
$\mathbf{A}$
and $\mathbf{B}$ be models of $T$; by
hypothesis there exists a model $\mathbf{C}$ of $T$ such that
$\mathbf{A}$ and $\mathbf{B}$ embed into $\mathbf{C}$. Since $T$ is model complete,
\[\mathbf{A}\preccurlyeq \mathbf{C}\mbox{ and }\mathbf{B}
\preccurlyeq\mathbf{C}\]
so that certainly $\mathbf{A}\equiv\mathbf{B}$ and  from
Lemma~\ref{complete},
$T$ is complete.
\end{proof}

By writing a formula $\varphi$ of $L$ as
$\varphi(\bar x)$, we indicate that the free variables of
$\varphi$ lie amongst those of $\bar x$. We can also write
$\varphi(\bar x;\bar y)$ to indicate that the free variables lie
amongst those of the distinct tuples $\bar x$ and $\bar y$.
A formula of the form $(\exists \bar x) \psi(\bar x;\bar y)$ for a
quantifier--free formula $\psi(\bar x;\bar y)$ is called
{\em existential}. A structure ${\mathbf A}$ in a class $\mathcal{K}$
of  $L$-structures
 is called  \em  existentially closed in
$\mathcal{K}$
 \em
if for every extension ${\mathbf B}\in \mathcal{K}$  of  ${\mathbf A}$
and every existential formula $(\exists \bar x)\varphi(\bar x;\bar a)$
with $\bar a\in A$,
 if ${\mathbf B}\models(\exists \bar x)\varphi(\bar x;\bar a)$, then
${\mathbf A}
\models(\exists \bar x)\varphi(\bar x;\bar a)$.

\begin{theorem} (Proposition 3.1.7)  \cite{changkeisler}
\label{modelcomplete}\label{modcom} A theory $T$
 is model complete if and only if
whenever ${\mathbf A},{\mathbf B}\models T$ and ${\mathbf A\subseteq
\mathbf B}$, then for any existential formula  $(\exists \bar
y)\psi(\bar y)$ of $L_A$,
$${\mathbf B}_A\models
(\exists \bar y)\psi(\bar y)\Rightarrow{\mathbf A}_A \models(\exists
\bar y)\psi(\bar y).$$
\end{theorem}

Let $\varkappa$ be a cardinal.
We recall that a theory $T$ in $L$ is {\em categorical in $\varkappa$},
or $\varkappa$-{\em categorical}, if $T$ has a model of cardinality
$\varkappa$, and any two models of cardinality $\varkappa$ are
isomorphic.
The next result is known as the {\em \L os-Vaught test}; its
proof is straightforward, relying upon the L\"{o}wenheim-Skolem-Tarski
theorems.

\begin{proposition}\label{catcomplete}
(Proposition 3.1.10) \cite{changkeisler} Suppose that $T$ is a
consistent theory with only infinite models, and that $T$ is
$\varkappa$-categorical for some $\varkappa\geq |L|$. Then $T$ is
complete.
\end{proposition}

When applying the notions of categoricity, completeness and model
completeness to a class of $L$-structures, we have to be a little careful.
A class $\mathcal{K}$ of $L$-structures is called \em categorical in
cardinality
$\varkappa$ \em or \em $\varkappa$-categorical
\em if all structures from $\mathcal{K}$ of cardinal $\varkappa$
are isomorphic. The class $\mathcal{K}$ is called \em categorical
\em if $\mathcal{K}$ is categorical in some cardinal $\varkappa\geq |L|$.

Let $\mathcal{K}$ be class of the
$L$-structures. We denote class of the
infinite structures of $\mathcal{K}$ by $\mathcal{K}_\infty$. The class
$\mathcal{K}$ is called \em complete (model complete)\em,
if the theory Th$(\mathcal{K}_\infty)$ of the infinite structures
of this class is complete (model complete).

\begin{lemma}\label{infinite} Let $\mathcal{K}$ be a class of
$L$-structures axiomatised by a theory $T$. For $n\in\mathbb{N}$ we let
$\varphi_n$ be the sentence
\[\varphi_n\leftrightharpoons (\exists x_1\hdots x_n)\bigwedge_{1\leq
i\neq j\leq n}(x_i\neq x_j)\]
and let $T_{\infty}$ be the deductive closure of
\[T\cup\{ \varphi_n\mid n\in\mathbb{N}\}.\]
Then $T_{\infty}$ axiomatises $\mathcal{K}_{\infty}$, so that
Th$\,(\mathcal{K}_{\infty})=T_{\infty}$.
\end{lemma}

\begin{lemma}\label{axclasses}\label{existsclosed}
 Let $\mathcal{K}$ be an axiomatisable class of
$L$-structures and $\varkappa$  an infinite cardinal,
$\varkappa\geqslant|L|$.

(1) If  $\mathcal{K}$ is closed under the union of increasing chains
then there is an existentially closed structure ${\mathbf A}\in \mathcal{K}$,
$|{\mathbf A}|=\varkappa$.

(2) If there is an infinite structure in $\mathcal{K}$ which is not
existentially closed, then there is a structure ${\mathbf A}\in \mathcal{K}$,
$|{\mathbf A}|=\varkappa$ and such that
${\mathbf A}$ is not existentially closed.
\end{lemma}

\begin{proof} (1) Note that for any $L$-structure ${\mathbf A}$ and
existential formula with parameters from ${\mathbf A}$ which is true in
${\mathbf A}$, this formula is true in any $L$-structure ${\mathbf B}$, ${\mathbf
A}\subseteq \mathbf B$. By Theorem~\ref{updownlst}, there is an $L$-structure
${\mathbf A}_0\in \mathcal{K}$ with
$|{\mathbf A}_0|=\varkappa$.

Enumerate the existential formulae of $L_{A_0}$ by $\{ (\exists \bar
x)\varphi_i(\bar x;\bar a)\mid i<\varkappa\}$ and define
$L$-structures $\mathbf B_i\in \mathcal{K}$, $0\leq i<\varkappa$
inductively as follows. We put $\mathbf B_0=\mathbf A_0$. If
$\mathbf B_i \models (\exists\bar x)\varphi_i(\bar x; \bar a)$, then
put $\mathbf B_{i+1}= \mathbf B_i$. If $\mathbf B_i\models \neg
(\exists\bar x)\varphi_i(\bar x;\bar a)$ and there is no $\mathbf
B\in \mathcal{K}$ with $\mathbf B_i\subset \mathbf B$ and $\mathbf
B\models (\exists\bar x)\varphi_i(\bar x;\bar a)$, then again we put
$\mathbf B_{i+1}=\mathbf B_i$. On the other hand, if we can find a
$\mathbf B\supset \mathbf B_i$ with $\mathbf B\models (\exists\bar
x)\varphi_i(\bar x;\bar a)$, put $\mathbf B_{i+1}'=\mathbf B$. By
Theorem~\ref{updownlst}, there is an elementary substructure
$\mathbf{B}_{i+1}$ of $\mathbf{B}_{i+1}'$ such that $B_i\subseteq
B_{i+1}$ and $|B_{i+1}|=\varkappa$. Certainly
$\mathbf{B}_{i+1}\models (\exists \bar x)\varphi_i(\bar x; \bar a)$.
For  a limit ordinal $\alpha$ we let $\mathbf
B_{\alpha}=\bigcup_{i<\alpha}\mathbf B_i$. Since $\mathcal{K}$ is
closed under unions of increasing chains, it is clear that $\mathbf
B_{\varkappa}$ is in  $\mathcal{K}$, $|B_{\varkappa}|=\varkappa$ and
$\mathbf B_{\varkappa}$ satisfies the condition:

(*) for any existential formula $\varphi$ with parameters from
structure ${\mathbf A}_0$ if $\varphi$ is true in some structure
 ${\mathbf
B}\in \mathcal{K}$, ${\mathbf B}\supseteq {\mathbf B}_{\varkappa}$, then $\varphi$ is true in
${\mathbf B}_{\varkappa}$.

Put $\mathbf{A_1}=\mathbf{B}_{\varkappa}$. Continuing this procedure
we receive an increasing chain of $L$-structures from $\mathcal{K}$
with  cardinality $\varkappa$:
$${\mathbf A}_0\subseteq {\mathbf A}_1\subseteq\ldots\subseteq
{\mathbf A}_n\subseteq\ldots$$
where each pair of structures ${\mathbf A}_n$, ${\mathbf A}_{n+1}$,
$n\in\omega$, satisfies  the condition (*) with  ${\mathbf
A}_0$ and ${\mathbf A}_1$ replaced by
 ${\mathbf A}_n$ and ${\mathbf A}_{n+1}$
accordingly. It is clear that the structure ${\mathbf A}=\bigcup\{{\mathbf
A}_n\mid n\in\omega\}\in \mathcal{K}$ is  existentially closed and $|{\mathbf
A}|=\varkappa$.

(2)
 Let ${\mathbf A}_0$ and ${\mathbf B}_0$ be infinite
structures from  $\mathcal{K}$ with
${\mathbf A}_0\subseteq {\mathbf B}_0$, for which there exists an
existential formula $(\exists \bar x)\varphi(\bar x;\bar a)$, where
$\bar a\subseteq A$, such that
${\mathbf A}_0\models\neg (\exists \bar x)\varphi(\bar x;\bar a)$ and
${\mathbf B}_0\models (\exists \bar x)\varphi(\bar x;\bar a)$. We
consider
$(\exists \bar x)\varphi(\bar x;\bar a)$ to be a sentence in
$L_{\bar a}$, and note that $\varkappa\geq |L_{\bar a}|$.
From Proposition~\ref{regultrafil}
and Theorem~\ref{filtprod} there is an ultrafilter $D$
over $\varkappa$
such that $|{\mathbf A}|\geq\varkappa$, where
${\mathbf A}={\mathbf A}_0^\varkappa/D$.
Then ${\mathbf A}\models\neg (\exists \bar x)\varphi(\bar x;\bar a)$ and
${\mathbf B}\models (\exists \bar x)\varphi(\bar x;\bar a)$, where
${\mathbf B}={\mathbf B}_0^\varkappa/D$,
moreover ${\mathbf A}\subseteq\mathbf B$.

By Theorem 3.1.6 of \cite{changkeisler} there exists an
elementary substructure
${\mathbf E}\preccurlyeq \mathbf A$
of  cardinality $\varkappa$. Clearly ${\mathbf E}\models\neg (
\exists \bar x)\varphi(\bar x;\bar a)$, ${\mathbf E}\in \mathcal{K}$,
 and ${\mathbf E}$ is not
existentially closed.
\end{proof}

The next result is known as Lindstr\"{o}m's Theorem.

\begin{theorem}\label{lindstrom} (Theorem 3.1.12) \cite{changkeisler}.
 Let a class $\mathcal{K}$ of
infinite $L$-structures be axiomatisable, categorical in some
infinite cardinality $\varkappa\geqslant|L|$ and closed under unions
of increasing chains. Then $\mathcal{K}$ is model complete.
\end{theorem}
\begin{proof} It is clear that the property of being existentially closed
is preserved by isomorphism. Since  all structures from
$\mathcal{K}$ of the cardinality $\varkappa$ are isomorphic then
from Lemma~\ref{axclasses}, all infinite structures from
$\mathcal{K}$ are existentially closed.
Theorem~\ref{modelcomplete}
 now gives that the class $K$ is model complete.
\end{proof}

\section{Completeness of $\mathcal{SF}$,
$\mathcal{P}$ and $\mathcal{F}r$}\label{comp}

We investigate here the monoids with axiomatisable
classes of
 strongly flat, projective and free $S$--acts, asking
for conditions under which these classes are
 complete and model complete. The results
of this section are all taken from \cite{stepanova}.

\begin{theorem}\label{fcomplete}
\label{complflat} Let $S$ be a commutative monoid. Suppose that
the  class $\mathcal{SF}$ is axiomatisable. Then the following conditions are
equivalent:

(1) the class $\mathcal{SF}$ is complete;

(2) the class $\mathcal{SF}$ is model complete;

(3) the class $\mathcal{SF}$ is categorical;

(4) $\mathcal{SF=F}r$;

(5) $S$ is an Abelian group.
\end{theorem}
\begin{proof} The implication (2)$\Rightarrow$(1) follows from
Lemma~\ref{modelcompleteimpliescomplete} and the
closure of the class $\mathcal{SF}$ under the coproducts.
The statement (4)$\Rightarrow$(2) follows from Theorem~\ref{lindstrom},
(3)$\Rightarrow$(1) is an immediate consequence of
Proposition~\ref{catcomplete}. It is clear that any two free
$S$-acts of cardinality $\alpha> |L|$ are isomorphic so that
  (4)$\Rightarrow$(3) follows.

(4)$\Leftrightarrow (5)$ This is immediate from Theorem 2.6 of \cite{kp}.

(1)$\Rightarrow$(5) Fix an arbitrary $a\in A$; it is enough to
prove that $aS=S$. Let  $\Phi$ be a uniform ultrafilter on $ \omega
$ and for $k\in\mathbb{Z}$ define $f_k\in S^\omega$ by
$$f_k(i)=\left\{
\begin{array}{ll}
   a^{k+i}&k+i>0; \\
  1&k+i\leq 0,
\end{array}
\right.$$ We will show that the left $S$--act $U =\bigcup_{k\in
Z}Sf_k/\Phi$, which is a subact of a left $S$--act $S ^\omega/\Phi$,
is strongly flat. Since $f_k/\Phi=af_{k-1}/\Phi$ then
$Sf_k/\Phi\subseteq Sf_{k-1}/\Phi$. Assume $r\, g_1/\Phi=s\,
g_2/\Phi$, where $r,s\in S$, $g_1/\Phi,g_2/\Phi\in U$. There exists
$k\in \mathbb{Z}$ such that $g_1/\Phi,g_2/\Phi\in S\, f_k/\Phi$,
i.e. $g_1/\Phi=t_1\, f_k/\Phi$, $g_2/\Phi=t_2\, f_k/\Phi$ for some
$t_1,t_2\in S$. Hence there exists $i\in\omega$ such that $i+k>0$,
$r\, g_1(i)=s\, g_2(i)$, $g_1(i)=t_1f_k(i)=t_1a^{k+i}$,
$g_2(i)=t_2f_k(i)=t_2a^{k+i}$. Thus $rt_1a^{k+i}=st_2a^{k+i}$,
$g_1/\Phi=t_1\, a^{k+i}f_{-i}/\Phi$, $g_2/\Phi=t_2a^{k+i}\,
f_{-i}/\Phi$. Thus $U$ satisfies condition (P); a minor adjustment
yields condition (E) also. Theorem~\ref{stronglyflat}  now implies
that $U$ is strongly flat.

As $S$ is a commutative monoid then
$$\coprod_{i\in\omega}U_i\models (\forall x)(\exists
y)(x=ay),$$ where $U_i$ are the copies of the left $S$--act $U$,
$i\in\omega$. Since the class $\mathcal{SF}$ is complete and
$S\in\mathcal{SF}$ then
$$\coprod_{i\in\omega}S_i\models (\forall x)(\exists
y)(x=ay),$$ where $S_i$ are the copies of the left $S$--act $S$,
$i\in\omega$. Therefore $aS=S$. \end{proof}

\begin{theorem}\label{projscomplete} Let $S$ be a monoid such that
$\mathcal{P}$ is an axiomatisable class. Then the
following conditions are equivalent:

(1) the class $\mathcal{P}$ is complete;

(2) the class $\mathcal{P}$ is model complete;

(3) the class $\mathcal{P}$ is categorical;

(4) $\mathcal{P}={\mathcal F}r$;

(5) $S$ is a group.
\end{theorem}
\begin{proof} We remark that if $S$ is a group, then again from
\cite{kp} we have that $\mathcal{SF}=\mathcal{P}=\mathcal{F}r$.
 As in the proof of Theorem~\ref{fcomplete} it is then
enough to prove the implication (1)$\Rightarrow$(5). The
axiomatizability of the class $\mathcal{P}$ and
Theorem~\ref{perfect} imply  that $S$ is a left perfect monoid and
$S$ therefore satisfies Condition $(M_R)$. Consequently, $S$ has a
minimal right ideal which from Proposition~\ref{perfectagain} $(4)$
is of the form $eS$ for some $e\in E$.
Proposition~\ref{perfectagain}
 $(2)$ gives that
left ideal $Se$ is also minimal. Clearly, $Se\in\mathcal{P}$.

Since $S$ is local,  $1$ is the {\em only} idempotent in the
$\ar$-class and in the $\el$-class of $1$.
Suppose $e\neq 1$ so that
$Se\subset S$ and $eS\subset S$.

For any $a\in S$ we put
\[X_a=\{x\in S\mid ex=a\}\]
so that  by Lemma~\ref{crfs} each set $X_a$
is a finite subset of $S$. Let $X_e\cap Se=\{a_1,\ldots,a_n\}$, $a_i\neq a_j$
($i\neq j$), and choose any $t\in Se$.  Notice that
$X_e\cap Se\subseteq L_e$.

We will show that $|X_{et}\cap Se|=n$. Clearly $a_it\in
X_{et}\cap Se$, $1\leq i\leq n$. Since $Se=Sa_i$ is a minimal left ideal
 Proposition~\ref{perfectagain}
gives that $a_iS$ is
minimal right ideal for $i\in\{ 1,\hdots ,n\}$.

Suppose $a_it=a_jt $ where $1\leq i,j\leq n$. Since the
ideals $a_iS$ and $a_jS$ are minimal right we have that $a_iS=a_jS$
and so $a_j=a_ik$ for some $k\in S$. Since $ea_i=ea_j=e$, we
deduce that
$ek=ea_ik=ea_j=e$, that is, $ek=e$. Since $Se=Sa_i$ then $a_i=a_ie$
and so
\[a_j=a_ik=a_iek=a_ie=a_i.\]
Hence $|\{ a_1t,\hdots ,a_nt\}|=n$.

Assume there exists $c\in X_{et}\cap Se$ such that $c\neq a_it$ for
any $i$, $1\leq i\leq n$. Since $c=ce$ and the left ideal $Se$ is
minimal, we have that $Se=Sce=Sc$, that is,  $Sc$ is a minimal left ideal.
Consequently $Sc=Sec$ and $c=lec$ for some $l\in S$. The minimality
of the ideal $eS$ implies the equality $eS=ecS$. Hence, there is
$d\in cS$ such that $ed=e$, that is, $d\in X_e$,
 and $d=cr$ for some $r\in S$. Suppose $d\in Se$. Then $d=a_i$ for some $i$, $1\leq i\leq n$.
The equalities $ecrt=edt=ea_it=et=ec$ imply $ecrt=ec$. Let us
multiply this equation by $l$ from the left. Then $crt=c$, i.e.
$c=a_it$. This contradicts the choice of $c$. Thus, $d\in(X_e\cap
cS)\setminus Se$. Since $d=cr$ we have $ecr=ed=e=ecre$. Let us multiply
this equation by $l$ from the left. Then $cr=cre$, that is $d\in
Se$, a contradiction.

Thus, $X_{et}\cap Se=\{a_1t,\cdots,a_nt\}$ and $a_it\neq a_jt$
($i\neq j$) for any $t\in Se$. So $Se\models\psi$ where
\[\psi\leftrightharpoons(\forall x)(ex=x\rightarrow (\exists (y_1,\hdots ,y_n)\bigg(
\bigwedge_{1\leq i\leq n}(ey_i=x)\wedge
 (ey=x\rightarrow\bigvee_{1\leq i\leq n}y=y_i)\bigg))).\]
 Since the class $\mathcal{P}$
is complete then $A=\coprod_{i\in\omega}Se_i\equiv
B=\coprod_{i\in\omega}S_i$, where $S_i$, $Se_i$, $i\in\omega$, are
copies of the left $S$--acts $S$ and $Se$ respectively,
$i\in\omega$. As $A\models\psi$ we must have that $B\models\psi$. In
particular, there are exactly $n$ solutions to $ey=e$; but $|X_e\cap
Se|=n$ and $1\in X_e\setminus Se$, a contradiction. So $e=1$ and (as
every principal right ideal contains a minimal one), $aS=S$ for all
$a\in S$. Since $S$ is a minimal right ideal,
Proposition~\ref{perfectagain} gives that it is a minimal left ideal
ideal and so $Sa=S$ for all $a\in S$. Consequently, $S$ is a group.
\end{proof}

The final result of this section follows
directly from the structure of  free left $S$--acts.

\begin{proposition}\label{freescomplete} Let $S$ be a monoid such that
$\mathcal{F}r$ is an axiomatisable class. Then $\mathcal{F}r$ is
categorical, complete and model complete.
\end{proposition}

\section{Stability}\label{stabintro}

Let $T$ be a consistent theory in the language $L$,
 let  ${\{x_i\mid
1\leq i\leq n \}}$ be a fixed set of variables and let
$L_n=L_{\{ x_1,\hdots ,x_n\}}$.
A consistent set of  sentences $p$ of $L_n$
is called an $n$--{\em type} of the language $L$. If
$p\,\cup T$ is consistent, that is, it has a model,
then $p$ is called an $n$--{\em type over}
$T$. If $p$ is complete, it is a {\em complete $n$-type}
and if in addition $T\subseteq p$ we
say that $p$ is a {\em complete $n$--type over $T$}. The set of all complete
$n$--types over $T$ is denoted by $S_n(T)$.

Let ${\mathbf A}$ be an $L$-structure, let $X\subseteq A$
and let $a\in A$.
The set
\[\mbox{tp}(a,X)=\{\varphi(x)\in L_X\mid {\mathbf A}\models\varphi(a)\}\]
 is
called the  {\em type of  $a$ over $X$}. Clearly
tp$(a,X)$ is a complete $1$--type over Th$({\mathbf A},x)_{x\in X}$;
we say that it is {\em realised} by $a$.
By $S_n({X})$ we denote $S_n($Th$({\mathbf A},x)_{x\in X})$. Often we will
write $S({X})$ instead $S_1({X})$.

A complete theory $T$ with no finite models is called {\em stable
 in a cardinal} $\varkappa$ or
$\varkappa$--{\em stable} if $|S(X)|\leq \varkappa$ for any model ${\mathbf
A}$ of the theory $T$ and any $X\subseteq A$ of cardinal
$\varkappa$. If  $T$  is $\varkappa$--stable for some
infinite $\varkappa$, then $T$ is called {\em stable}. If  $T$
is $\varkappa$--stable for all $\varkappa\geq2^{|T|}$, then $T$ is
called {\em superstable}. An {\em unstable}
theory is one which is not stable!

Morley proved \cite{morley} that, if a countable theory $T$ is $\omega$-stable,
then it stable in every
infinite cardinality.
If an arbitrary theory $T$ in a language $L$ is $\omega$-stable, then
$|S_n(T)|\leqslant\omega$ for all $n\in\omega$.
It follows that $T$ is essentially countable, in the following
sense. There must be a sublanguage
$L'\subseteq L$, such that  $|L'|=\omega$
and for each formula $\varphi$ of $L$ there is a formula
$\varphi'$ of $L'$ such that for any $L$-structure $\mathbf A$ of
$L$ with  $\mathbf A\models T$, we have that
$\mathbf A\models \varphi$ if and only if $\mathbf A\models \varphi'$.
Consequent upon the result of Morley, if $T$ is
$\omega$-stable, then it is stable in every
infinite cardinality.

The notion of the monster model of a complete theory is a useful
tool in our investigations. In order to define such a model, we need the
notions of saturation and homogeneity.
An $L$-structure $\mathbf A$ is $\varkappa$-{\em homogeneous}
for a cardinal
$\varkappa$ if for any $X\subseteq A$ with $|X|<\varkappa$, any
map $f:X\rightarrow A$ with
\[\mbox{tp}(\mathbf A,x)_{x\in X}=\mbox{tp}(\mathbf A,f(x))_{x\in X}\]
can be extended to an automorphism of $\mathbf A$.
An $L$-structure $\mathbf A$ is $\varkappa$-{\em saturated} for a
cardinal
$\varkappa$ if for any $X\subseteq A$ with $|X|<\varkappa$, every
$p\in S(\mbox{Th}(\mathbf A,x)_{x\in X})$ is realised in $\mathbf
A$.

Suppose now that $T$ is a complete theory in $L$. We may find a
cardinal $\overline
{\varkappa}$
 greater than all others under consideration and a
$\overline{\varkappa}$-saturated and
 $\overline{\varkappa}$-homogeneous model $\mathbf M$ of $T$.
The
 convention is that all models of $T$ will be elementary
substructures of $\mathbf{M}$ of cardinality strictly less than
$\overline{\varkappa}$, and all sets of parameters will be subsets of
$M$, with again,  cardinality strictly less that $\overline{\varkappa}$.
With this convention, if $\mathbf{A}$ is a model of $T$,
$X\subseteq A$ and $a\in A$, then
\[\mbox{tp}(a,X)
=\{ \varphi(x)\in L_X\mid \mathbf{M}\models \varphi(a)\}.\]

The model $\mathbf M$ is called
the {\em monster model} of $T$. Justification of the use of the monster model, and the following result,
may be found in standard stability theory texts, such
as \cite{pillay}.

\begin{lemma}\label{typesandisos} Let $T$ be a complete theory in $L$
with monster model $\mathbf M$
and
let $X$ be a subset of $M$.  Suppose also that
 $\bar a,\bar a'\subseteq M$. Then
\[\mbox{tp}(\bar a,X)=
\mbox{tp}(\bar a',X)\]
 if and only if there exists an
automorphism of  $\mathbf M$, which acts identically
on $X$ and maps $\bar a$ into $\bar a'$.
\end{lemma}

Let $S$ be a monoid and let $\mathcal{K}$
be a class of  left $S$--acts. Then $S$ is called a
$\mathcal{K}$--{\em stabiliser} ($\mathcal{K}$--{\em superstabiliser},
 $\mathcal{K}$--$\omega$--{\em stabiliser}) if
Th$(A)$ is stable (superstable, $\omega$--stable)
for any infinite left $S$--act $A\in \mathcal{K}$. If $\mathcal{K}$ is the class
of all left $S$-acts, then a $\mathcal K$--stabiliser
($\mathcal K$--superstabiliser, $\mathcal K$--$\omega$--stabiliser)
is referred to more simply as a
stabiliser (superstabiliser, $\omega$--stabiliser).

\section{Superstability of ${\mathcal{SF}}$,
${\mathcal P}$ and ${\mathcal F}r$}\label{stability}

Now we begin to consider the stability questions for $S$--acts. The
results of this section are all taken from \cite{ste2}.

\begin{lemma}\label{cfrsforsf} Let  $S$ be a monoid satisfying (CFRS). Then for
any $A\in\mathcal{SF}$, $a\in A$, $s\in S$
$$|\{x\in A\mid sx=a\}|\leq n_s.$$
\end{lemma}

\begin{proof} Assume there
exists $A\in\mathcal{SF}$, $s\in S$ and $a_0,\ldots,a_{n_s}\in A$
such that $sa_i=sa_j$, $a_i\neq a_j$ $(i\neq j)$. By induction on
$n\leq n_s$ we will show that there exist $b\in A$,
$r_0,\ldots,r_n\in S$ such that $a_i=r_ib$, $sr_i=sr_j$ for any
$i,j\in \{0,\ldots,n\}$. Let $n=1$. Since $A\in\mathcal{SF}$
there are $r_0^\prime,r_1^\prime\in S$ and $b^\prime\in A$ such that
$sr_0^\prime=sr_1^\prime$, $a_0=r_0^\prime b^\prime$ and
$a_1=r_1^\prime b^\prime$. Suppose there exist
$r_0^{\prime\prime},\ldots,r_{n-1}^{\prime\prime}\in S$ and
$b^{\prime\prime}\in A$ such that
$sr_i^{\prime\prime}=sr_j^{\prime\prime}$ and
$a_i=r_i^{\prime\prime}b^{\prime\prime}$ for any
$i,j\in\{0,\ldots,n-1\}$. As $A\in\mathcal{SF}$
and $sr_0''b''=sa_n$,  there exist
$r,r_n\in S$ and $b\in A$ such that $sr_0^{\prime\prime}r=sr_n$,
$b^{\prime\prime}=rb$ and $a_n=r_nb$. Let $r_i=r_i^{\prime\prime}r$
$(0\leq i\leq n-1)$. Then
$a_i=r_i^{\prime\prime}b^{\prime\prime}=r_i^{\prime\prime}rb=r_ib$,
$sr_i=sr_i^{\prime\prime}r=sr_j^{\prime\prime}r=sr_j$ for any
$i,j\in\{0,\ldots,n-1\}$. Thus there exist $r_0,\ldots,r_{n_s}\in S$
such that $r_i\neq r_j$ $(i\neq j)$ and $sr_i=sr_j$ for any
$i,j\in\{0,\ldots,n_s\}$, contradicting the fact that
$S$ satisfies (CFRS).
\end{proof}

\begin{lemma}\label{cfrsprec} Let  $S$ be a monoid satisfying (CFRS),
$B\in\mathcal{SF}$, $B\preccurlyeq C$. Then $\bigcup_{c\in
C\backslash B}Sc\cap B=\emptyset$.
\end{lemma}

\begin{proof} Let the given conditions hold and suppose that
$b\in\bigcup_{c\in C\setminus B}Sc\cap B$. Then there exists $c_1\in
C\backslash B$ and $s\in S$ such that $b=sc_1$. The formula
$$(\exists x_1\ldots x_k)(\bigwedge\limits_{1\leq i<j\leq k}x_i\neq
x_j\wedge\bigwedge\limits_{1\leq i\leq k}y=sx_i)$$ will be denoted by
$\varphi_k(y)$.

Clearly $C\models \varphi_1(b)$. We show by  induction on $k$
 that $C\models\varphi_k(b)$ for
all $k\geq 1$. Suppose that $k\geq 1$ and $C\models\varphi_k(b)$,
that is $b=sc_i$, $1\leq
i\leq k$, where $c_1,\hdots ,c_k$ are distinct elements of $C$. Since
$B$ is an elementary substructure of $C$ then $B\models\varphi_k(b)$,
that is $b=sb_i$, $1\leq i\leq k$, where $b_1,\hdots, b_k$ are
distinct elements of $B$. Since also $sc_1=b$ and
$c_1\notin B$ we have that
$C\models\varphi_{k+1}(b)$. Therefore $C\models\varphi_k(b)$ for any
$k\geq 1$. Since $B$ is an elementary substructure of $C$ then
$B\models\varphi_k(b)$ for any $k\geq 1$, contradicting
Lemma~\ref{cfrsforsf}.
\end{proof}

\begin{lemma}\label{sftype} Let $S$ be a monoid $S$ satisfying (CFRS),
$B\in\mathcal{SF}$, $B\preccurlyeq M$ where $M$ is the monster model
of Th$(B)$, and let $c_1,c_2\in M\backslash B$. Then
\[\mbox{tp}(c_1,B)=\mbox{ tp}(c_2,B)
\Leftrightarrow \mbox{tp}(c_1,\emptyset)=\mbox{ tp}(c_2,\emptyset).\]
Consequently, for any subset $B'\subseteq B$, we have
\[\mbox{tp}(c_1,B')=\mbox{ tp}(c_2,B')
\Leftrightarrow \mbox{tp}(c_1,\emptyset)=\mbox{ tp}(c_2,\emptyset).\]
\end{lemma}

\begin{proof} Let the conditions of the Lemma hold.
Necessity is obvious.
To prove  sufficiency, suppose that tp$(c_1,\emptyset)=
$ tp$(c_2,\emptyset)$. From Lemma~\ref{typesandisos}
there exists an automorphism $\phi:M\rightarrow M$ such that
$\phi(c_1)=c_2$. According to Lemma~\ref{cfrsprec},
$M$ is the disjoint union of $B$, $C_1$ and $C_2$ and $D$,
where for $i=1,2$, $C_i$
is the connected component containing $c_i$,
and $D=M\setminus (B\cup C_1\cup C_2)$.

 Since $S$-morphisms preserve
the relation $\sim$ we must have that
$\phi:C_1\rightarrow C_2$ is an $S$-isomorphism. If
$C_1=C_2$ then we define $\psi:C\rightarrow C$ by
\[\psi|_{C_1}=\phi|_{C_1} \mbox{ and }\psi|_{B\cup D}=I_{B\cup D}\]
and if $C_1\cap C_2=\emptyset$ we define $\psi$ by
\[\psi|_{C_1}=\phi|_{C_1},\psi|_{C_2}=\phi^{-1}|_{C_2}
\mbox{ and }\psi|_{B \cup D}=I_{B\cup D}.\]
Clearly in either case $\psi$ is an
$S$-automorphism of the left $S$--act $M$. Since  $\psi(c_1)=c_2$
 and $\psi|_B=I_B$ we have that tp$(c_1,B)=$ tp$(c_2,B)$.
\end{proof}

\begin{theorem}\label{sfsuperst} Let $S$ be a monoid such that
$\mathcal{SF}$ is axiomatisable and $S$ satisfies (CFRS). Then $S$ is
$\mathcal{SF}$-superstabiliser.
\end{theorem}
\begin{proof} Suppose the monoid $S$ satisfies (CFRS),
$T=$ Th$(A)$  is the complete theory of some  infinite strongly flat
left $S$--act $A$, $M$ is the monster model of $T$ and $B\subseteq
M$ with $|B|=\varkappa\geq2^{|T|}$. By Theorem~\ref{updownlst},
there is a left $S$-act $B'\preccurlyeq M$ with $B\subseteq B'$ and
$|B'|=\varkappa$. Since $\mathcal{SF}$ is axiomatisable, $M$ and
$B'$ are strongly flat. According to Lemma~\ref{sftype},
$|\{\mbox{tp}(x,B)\mid x\in M\setminus B'\}|=|\{\mbox{tp}
(x,\emptyset)\mid x\in M\setminus B'\}|\leq 2^{|T|}\leq\varkappa$.
Furthermore, $|\{\mbox{tp} (x,B)\mid x\in B'\}|= \varkappa$. Thus,
$|S(B)|\leq\varkappa$ and the theory $T$ is superstable.
\end{proof}

\begin{corollary}\label{cancellative} If $S$ is a left cancellative monoid such that
$\mathcal{SF}$ is axiomatisable, then $S$
is an $\mathcal{SF}$--superstabiliser.
\end{corollary}

\begin{proof} Let $S$ be a left cancellative monoid, $s\in S$. Since
$s$ is a left 1-cancellable element then $|\{x\in S\mid sx=t\}|=1$
for any $t\in S$. So $S$ satisfies (CFRS) and by
Theorem~\ref{sfsuperst}, $S$ is an $\mathcal{SF}$-superstabiliser.
\end{proof}

For an example of a monoid $S$ satisfying the conditions of
Corollary~\ref{cancellative} we can take the free monoid $X^*$ on a set
$X$, which is certainly left cancellative.
Since $X^*$ is also right cancellative, $r(s,t)=\emptyset$ if
$s\neq t$, and if $s=t$ then $r(s,t)=X^*$, which is principal.
If neither $s$ nor $t$ is a prefix of the other, then
$R(s,t)=\emptyset$; on the other hand if $s=tw$ then
$R(s,t)=(1,w)X^*$ and dually if $s$ is a prefix of $t$.

\begin{lemma}\label{axsfandacfrs} Let
$S$ be a monoid such that $\mathcal{SF}$ is axiomatisable and
 Condition $(A)$ holds. Then  $S$ satisfies
(CFRS).
\end{lemma}

\begin{proof} Suppose to the contrary that there is $s_1\in S$ such
that for any $i\in\omega$ there exists $b_i\in S$
such that:
$$T_i=|\{x\in S\mid s_1x=b_i\}|\geq i.$$
We let $x_{i,1},\hdots ,x_{i,i}$ be distinct elements of $T_i$.
Let $D$ be a uniform ultrafilter over $\omega$,
let $S_0=S^{\omega}/D$,
$\bar{a}\in S_0$,
$\bar{a}=a/D$, $a(i)=b_i$ for any $i\in\omega$.
For $i\in\omega$ we define $x_i\in S^{\omega}$
by
\[x_i(j)=\bigg\{
\begin{array}{ll}
1&j<i\\
x_{j,i}&j\geq i\end{array}\]
and put $\overline{x_i}=x_i/D$. It is clear that
the $\overline{x_i}$'s are distinct and so
$|\{x\in S_0\mid s_1x =\bar a\}|\geq\omega$. In view of the
axiomatizability of the class $\mathcal{SF}$  we have $S_0\in
\mathcal{SF}$. Now choose a cardinal $\alpha>|S|$. According to Proposition
\ref{regultrafil} and Theorem \ref{filtprod} we can choose an
ultrafilter $\Phi$ over $\alpha$ such that $|\{x\in S_0\mid
s_1x=\bar{a}\}^{\alpha}/\Phi|>|S|$. Denote $S_0^{\alpha}/\Phi$ by
$S_1$, so that as
$\mathcal{SF}$ is axiomatisable, $S_1$
is strongly flat,  put $a_0=a^{\prime}/\Phi$, where
$a^{\prime}(\beta)=\bar{a}$ for any $\beta<\alpha$, and let
 $A_1=\{x\in S_1\mid s_1x=a_0\}$, so that $|A_1|>|S|$. As $|Sa_0|\leq|S|$ there exists $a_1\in
A_1$ such that $Sa_0\subset Sa_1$.

Let $k\in\mathbb{N}$. Assume that for $0<i<k$
the sets $A_i\subseteq S_1$, elements $a_i\in A_i$ and $s_i\in S$
are defined
such that $Sa_{i-1}\subset Sa_i$,
 $A_i=\{x\in S_1\mid
s_ix=a_{i-1}\}$ and $|A_i|>|S|$. Let us define $A_k\subseteq S_1$,
$a_k\in A_k$, $s_k\in S$ such that $A_k\subseteq\{x\in S_1\mid
s_kx=a_{k-1}\}$, $Sa_{k-1}\subset Sa_k$ and $|A_k|>|S|$. From
Theorem \ref{sf}, since it is certainly not empty,
$R(s_{k-1},s_{k-1})=\{( x,y)\in S^2\mid
s_{k-1}x=s_{k-1}y\}=\bigcup_{0\leq i\leq m}(u_i,v_i) S$ for some
$m\in\omega$, $u_i,v_i\in S$ $(0\leq i\leq m)$. Clearly
$s_{k-1}a_{k-1}=s_{k-1}b$ for all $b\in A_{k-1}$ so there exists
$i$, $0\leq i\leq m$, such that $|\{x\in S_1\mid u_ix=a_{k-1}$,
$v_ix\in A_{k-1}\}|>|S|$. Let $s_k=u_i$, $A_k=\{x\in S_1\mid
s_kx=a_{k-1}\}$. As $|A_k|>|S|$ and $|Sa_{k-1}|\leq|S|$ then there
exists $a_k\in A_k$ such that $Sa_{k-1}\subset Sa_k$. Thus, there is
the ascending chain of cyclic $S$-subacts $Sa_i$ $(i\in\omega)$ of
the left $S$--act $S_1$, contradicting our  hypothesis that
Condition (A) holds.
\end{proof}

From Lemma \ref{axsfandacfrs} and Theorem \ref{sfsuperst}
our next result immediately follows.

\begin{corollary}\label{axsfanda} For a monoid $S$, if the
class $\mathcal{SF}$ is axiomatisable
and  $S$ satisfies Condition $(A)$, then $S$ is an
$\mathcal{SF}$--superstabiliser.
\end{corollary}

\begin{corollary}\label{psuperst} If the class $\mathcal P$
is axiomatisable for a monoid $S$,
then $S$ is a $\mathcal P$--superstabiliser.
\end{corollary}

\begin{proof} Let the class $\mathcal P$ be axiomatisable. From Theorem
\ref{projax}, $\mathcal{SF}$  is axiomatisable and $S$ is a
left perfect monoid
so that, according to  Theorem \ref{perfect},  $S$ satisfies
Condition $(A)$ and $\mathcal{SF}=
\mathcal{P}$. Now Corollary  \ref{axsfanda} yields
that $S$ is a
$\mathcal{P}$--superstabiliser.
\end{proof}

\begin{corollary} If the class ${\mathcal F}r$ is axiomatisable
for a monoid $S$,
then $S$ is an ${\mathcal F}r$--superstabiliser.
\end{corollary}

\begin{proof} Let ${\mathcal F}r$ be axiomatisable. From
Theorem \ref{freesax} the class $\mathcal P$  is axiomatisable. Now
Corollary \ref{psuperst} says that $S$ is $\mathcal P$--superstabiliser,
so that certainly
$S$ is an ${\mathcal F}r$--superstabiliser.
\end{proof}

\section
{ $\omega$-stability of ${\mathcal{SF}}$, ${\mathcal P}$ and
 ${\mathcal F}r$}\label{morestability}

All results from this section are again taken from \cite{ste2}.

\begin{lemma}\label{1} If $\theta$ is a left congruence of a monoid $S$
then $1/\theta$ is a submonoid of $S$.
\end{lemma}

\begin{proof} If  $u,v\in 1/\theta$ then $1\, \theta\, u$ so that
as $\theta$ is left compatible,
\[vu\,\theta\, v\, 1= v\,\theta\, 1.\]
\end{proof}

\begin{lemma}\label{2} Let $\theta_1,\theta_2$ be strongly flat left
congruences of a monoid $S$. Then
$$\theta_1=\theta_2\mbox{ if and only if }1/\theta_1=1/\theta_2. $$
\end{lemma}

\begin{proof} The necessity is obvious. Suppose now that
$1/\theta_1=1/\theta_2$.
Let $u,v\in S$. If $u\,\theta_1\, v$ then from Theorem
~\ref{cyclicproj} there exists $s\in S$
such that $s\, \theta_1\,  1$ and $us=vs$.  Hence
$s\,\theta_2\, 1$ and again from Theorem~\ref{cyclicproj}, $u\,\theta_2\, v$.
\end{proof}

\begin{lemma}\label{3} Let $S$ be a monoid such that
the class of left $S$-acts satisfying Condition (E) is axiomatisable.
 Let $M$ be a left $S$-act satisfying Condition
(E)  and let $Sa$ be a
connected component of $M$. Then the relation $\theta_a=\{(
s,t)\in S^2\mid sa=ta\}$ is a strongly flat left congruence of the monoid $S$
and the mapping  $\phi:Sa\rightarrow S/\theta_a$ given by
$\phi(sa)=s/\theta_a$ $(s\in S)$ is an isomorphism of  left
$S$--acts.
\end{lemma}

\begin{proof} It is obvious that the relation $\theta_a$ is a left
congruence of  $S$: we claim that $\theta_a$ is flat.
Suppose that $s\theta_a t$, so
that  $sa=ta$. From Proposition~\ref{e} we have that
$r(s,t)\neq \emptyset$ and the following sentence is true in $M$
\[ (\forall x)(sx=tx\rightarrow(\exists y)\bigvee\limits_{0\leq
i\leq k}x=u_iy)\]
where $\{ u_0,\hdots ,u_k\}$ is a set
of generators of the right ideal $r(s,t)$. Then $a=u_ib$ for
some $i$, $0\leq i\leq k$, and $b\in M$. Since $Sa$ is a connected
component of a left $S$-act $M$, then $Sa=Sb$, i.e. $b=ka$ for some
$k\in S$. Consequently $a=u_ika$, that is, $1\theta_au_ik$.
Furthermore the equation $su_i=tu_i$ implies the equation
$s(u_ik)=t(u_ik)$. According to Corollary~\ref{sfcong}, $\theta_a$
 is a
flat left congruence of the monoid $S$. The mapping
$\phi:Sa\rightarrow S/\theta_a$, where
$\phi(sa)=s/\theta_a$, is obviously an $S$-isomorphism.
\end{proof}

\begin{lemma}\label{4} If for a monoid $S$
the class $\mathcal{SF}$ is axiomatisable and
$S$ satisfies Condition $(A)$,
then any $M\in\mathcal{SF}$ is
 a coproduct  cyclic left $S$--acts.
\end{lemma}

\begin{proof} Suppose that $\mathcal{SF}$ is axiomatisable and
$S$ satisfies Condition $(A)$. Let $M\in\mathcal{SF}$ and write
$M=\coprod_{i\in I}M_i$, where $M_i$ is a connected component of $M$
$(i\in I)$. It is clear that $M_i\in\mathcal{SF}$ $(i\in I)$. Assume
that $M_i$ is not a cyclic left $S$--act for some $i\in I$. Then for
any $a\in M_i$ there is $b\in M_i$ such that $Sa\subset Sa\cup Sb$.
It is easy to see that for elements $u,v$ of a connected component
of a strongly flat $S$-act there is an element $w$ such that $Su\cup
Sv\subseteq Sw$. Since $M_i\in\mathcal{SF}$ then there exists $c\in
M_i$ such that $Sa\cup Sb\subseteq Sc$, so that $Sa\subset Sc$. Thus
there exists $a_j\in M_i$ $(j\in\omega)$ such that $Sa_j\subset
Sa_{j+1}$, contradicting the fact that $S$ satisfies Condition
$(A)$.
\end{proof}

For a subset $X$ of a monoid $S$ we put $\rho_X=\rho(X\times X)$,
that is, $\rho_X$ is the left
congruence on $S$ generated by $X\times X$.

\begin{lemma}\label{V} Let $T$ be a submonoid of a monoid $S$.
Then
$T$ is a class of $\rho_T$  if and only if
$T$ is a right unitary submonoid.
\end{lemma}

\begin{proof} Suppose first that $T$ is a class
of $\rho_T$ and $st\in T$, where
$s\in S$, $t\in T$. Since $t\,\rho_T\, 1$,  then $st\, \rho_T s$ and $s\in
T$. So $T$ is a right unitary submonoid.

Conversely, suppose that  $T$ is  a right unitary submonoid of
$S$ and $x\,\rho_T\, y$ where $ x\in T$. We claim that $y\in T$.
 By Proposition~\ref{cong}, $x=y$ (so that certainly $y\in T$), or there exist $n\in \omega$,
$t_0,\ldots,t_{2n+1}\in T$, $s_0,\ldots,s_n\in S$ such that
\begin {equation}
\tag {v} x=s_0t_0,\,
s_0t_1=s_1t_2,\hdots ,s_it_{2i+1}=s_{i+1}t_{2i+2},\hdots ,s_{n}t_{2n+1}=y
\end {equation}
for any $i$, $0\leq i\leq n-1$. By the induction on $i$ we prove
that $s_i\in T.$ Note that $s_0t_0,t_0\in T$ give that
$s_0\in T$ since $T$ is right unitary. For $i> 0$, if
 $s_{i-1}\in T$, then the
equality $s_{i-1}t_{2 i-1}=s_it_{2i}$ implies $s_{i}t_{2i}\in
T$. Since  $T$ is a right unitary submonoid of $S$ then
$s_i\in T$. So $s_n\in T$ and thus $y=s_nt_{2n+1}\in T$.
\end{proof}

\begin{lemma}\label{VI} Let $T$ be a right unitary submonoid of
$S$. Then $T$ is right collapsible if and only if
the left congruence $\rho_T$ is strongly flat.
\end{lemma}

\begin{proof} Let  $T$ be a right unitary submonoid of $S$, and
suppose first that $T$ is right collapsible and
$x\,\rho_T\, y$ where $x,y\in S$. If $x=y$ then $x\, 1=y\, 1$.
Otherwise there exist $n\in \omega$,
$t_0,\ldots,t_{2n+1}\in T$, $s_0,\ldots,s_n\in S$ such that (v)
holds. From Lemma~\ref{rightcollapsible} there exists $r\in T$ such that
$t_ir=t_jr$, $0\leq i,j\leq 2n+1$. Hence
\[xr=s_0t_0r=s_0t_1r=s_1t_2r=\hdots =s_nt_{2n+1}r=yr\]
and certainly $r\,\rho_T\, 1$. From Corollary~\ref{sfcong}, $\rho_T$
is a strongly flat left congruence.

Conversely, assume that $\rho_T$ is a strongly flat left congruence and $x,y\in
T$. Then $x\,\rho_T\, y$ and from Corollary~\ref{sfcong},
and Lemma~\ref{V} there exists
$z\in T$ such that $xz=yz$. Hence $T$ is right collapsible.
\end{proof}

We will write $CU^S$ for the set of all right collapsible and right
unitary submonoids of a monoid $S$.

\begin{theorem}\label{sfomegastab} Let $S$ be a monoid
with $|S|\leq\omega$. Suppose that the class $\mathcal{SF}$ is
axiomatisable
 and $S$ satisfies Condition $(A)$. Then $S$ is an
$\mathcal{SF}$--$\omega$--stabiliser if and only if
$|CU^S|\leq\omega$.
\end{theorem}

\begin{proof} Suppose the given hypotheses hold.
Assume first that $S$ is an $\mathcal{SF}$--$\omega$--stabiliser.
Let
$C$ denote the set of strongly flat left congruences on $S$, so that
$C\neq\emptyset$ as the equality relation $\iota$ lies in $C$. Put
\[A=\coprod\{S/\rho\mid\rho \in C\}\mbox{ and }B=\coprod\limits_{i\in\omega}A_i,\] where the $A_i$'s are
disjoint copies of the left $S$--act $A$ and  $1_i/\rho\in A_i$ is a
copy of $1/\rho\in A$ $(i\in\omega)$. It is clear that $A$ and $B$
lie in $\mathcal{SF}$. Put $T=$ Th$(B)$.

By assumption,  the theory $T$ is $\omega$--stable and
$|S(\emptyset)|\leq\omega$. Let $U\in CU^S$ so that from
Lemma~\ref{VI}, $\rho_U$ is a strongly flat left congruence of $S$,
that is, $S/\rho_U\in\mathcal{SF}$, and from  Lemma~\ref{V},
$U=1/\rho_U$. Suppose $\alpha:{CU^S}\rightarrow$ $S(\emptyset)$ is
the mapping such that $\alpha(U)=$ tp$(1_0/\rho_U,\emptyset)$. We
claim that $\alpha$ is an injection. Let $U,V\in CU^S$ with  $U\neq
V$ and without loss of generality choose $t\in U\setminus V$. Then
$t\cdot1_0/\rho_U=1_0/\rho_U$ and $t\cdot1_0/\rho_V\neq1_0/\rho_V$.
Hence $tx=x$ lies in tp$(1_0/\rho_U,\emptyset)$ but not in
tp$(1_0/\rho_V,\emptyset)$. Hence $\alpha$ is an injective mapping
and consequently
  $|CU^S|\leq\omega$.

Conversely, assume that $|CU^S|\leq\omega$.
Let $s,t\in S$. For
$d\in D\in \mathcal{SF}$, where $Sd$ is
connected, we denote the relation $\{(s,t)\in
S^2\mid sd=td\}$  by $\theta_d$; clearly, $\theta_d$ is a left
congruence on $S$. In view of
the axiomatisability of $\mathcal{SF}$, Theorem~\ref{sf} gives that the set
$r(s,t)$ is empty or is finitely generated as a right ideal of $S$.
Lemma~\ref{3} now gives that the
 relation $\theta_d$ is a strongly flat left congruence of
$S$.

Let  $F\in\mathcal{SF}$ with $|F|\geq\omega$ and let $T=$ Th$(F)$; we
must show that $T$ is $\omega$-stable. To this end, let $M$ be
the monster model of $T$, so that as $\mathcal{SF}$ is
axiomatisable, certainly $M$ is a strongly flat left $S$-act. Let
$A\subseteq M$ with $|A|\leq \omega$; by Theorem~\ref{updownlst}, there is a model
$B$ of $T$ with $A\subseteq B$ and $|B|=\omega$. By
construction of $M$,
$B\preccurlyeq M$.

Let $c_1,c_2\in M$. According to
 Lemma~\ref{4}, $M$ is a coproduct of cyclic left $S$-acts.
Hence there
exists $d_1, d'\in M$ such that $c_1\in Sd_1$, $c_2\in Sd^{\prime}$
with $Sd_1,Sd^{\prime}$   connected components of $M$. It is
clear that either $Sd_1=Sd^{\prime}$, or $Sd_1\cap
Sd^{\prime}=\emptyset$. We claim that for $c_1,c_2\in C\setminus B$,
tp$(c_1,A)=$ {tp}$(c_2,A)$ if and only if
\begin {equation}
\tag{vi}\exists d_2\in M,t\in S:
\;Sd_2=Sd^{\prime},\;1/\theta_{d_1}=1/\theta_{d_2},\;c_1=td_1,\;c_2=td_2.
\end {equation}

Let tp$(c_1,A)=$ tp$(c_2,A)$. Then there exists an $S$-automorphism
$\phi:C\rightarrow C$ such that $\phi(c_1)=c_2$, and  $\phi|_A=I_A$.
Putting $d_2=\phi(d_1)$ we have that $Sd_2$ is a connected component
containing $c_2$, so that $Sd_2=Sd'$. Now $c_1=td_1$ for some $t\in
S$, so that $c_2=td_2$. Furthermore
$$ud_1=vd_1\Leftrightarrow ud_2=vd_2$$ for any $u,v\in S$, i.e.
$\theta_{d_1} =\theta_{d_2}$, in particular
$1/\theta_{d_1}=1/\theta_{d_2}$. Hence (vi) holds.

Suppose conversely that (vi) holds. From Lemma~\ref{1},
 the sets $1/\theta_{d_1},
1/\theta_{d_2}$ are submonoids of  $S$. Since
$\theta_{d_1},\theta_{d_2}$ are strongly flat left congruences and
$1/\theta_{d_1}=1/\theta_{d_2}$,  Lemma~\ref{2} gives that
$\theta_{d_1} =\theta_{d_2}$ and so
$S/\theta_{d_1}=S/\theta_{d_2}$. From Lemma~\ref{3} there is an
isomorphism $\phi:Sd_1\rightarrow Sd_2$ of left $S$--acts such that
$\phi(d_1)=d_2$. Since $c_1=td_1$ and $c_2=td_2$, then
$\phi(c_1)=c_2$. From Lemma~\ref{axsfandacfrs},  $S$ has Condition
(CFRS) so that by Lemma~\ref{cfrsprec}, $Sd_i\cap B=\emptyset$ for
$i\in \{1,2\}$. Define an automorphism $\psi:M\rightarrow M$ as
follows: $\psi|_{Sd_1} =\phi|_{Sd_1}$,
$\psi|_{Sd_2}=\phi^{-1}|_{Sd_2}$ (unless $Sd_1=Sd_2$) and
$\psi|_{M\setminus (Sd_1\cup Sd_2)}=I$. Since in either case
$\psi(c_1)=c_2$, and $\psi|_A=I_A$, then tp$(c_1,A)=$ tp$(c_2,A)$.

Thus the type of any element $m\in M\setminus B$ over $A$ is
determined by some element of $S$ and by some flat congruence
$\theta_d$, where $Sd$ is a connected component and $m\in Sd$.

If $Sm$ is a connected component of $M$, then as
noted above, $\theta_m$ is a strongly flat left
congruence of  $S$. Let
$U=1/\theta_m$. We know that $U$ is a submonoid of $S$; if
$t, st\in U$ then $tm=m=stm$ so that $sm=m$ and $m\in U$. Hence
$U$ is right unitary.  On the other hand, if $p,q\in U$, then
$pm=m=qm$. As $M$ is strongly flat we have that $pr=qr$ for some $r\in
S$
and
with $m=rk$ for some $k\in Sm$. Since $Sm$ is a connected component we
deduce that $k=r'm$ for some $r'\in S$. Then $prr'=qrr'$ and
$m=rr'm$, so that $rr'\in U$ and $U$ is right collapsible, that is,
$U\in CU^S$. From Lemma~\ref{2},
 \[|\{\theta_c\mid Sc\mbox{
is a connected component of }M\}|\leq|CU^S|\leq\omega.\]
Since
\[|\{\mbox{tp}(b,A)\mid b\in B\}|\leq |B|=\omega\]
we deduce that  $|S(A)|\leq\omega$ as required.
\end{proof}

From Theorem~\ref{sf} and Theorem~\ref{sfomegastab} we have

\begin{corollary} If $S$ is a finite monoid then $S$ is an
$\mathcal{SF}$--$\omega$--stabiliser.
\end{corollary}

\begin{corollary} If $S$ is a countable group then $S$ is an
$\mathcal{SF}$--$\omega$--stabiliser.
\end{corollary}

\begin{proof} Let $S$ be a countable group.  As remarked in
\cite{gould1}, both
$\mathcal{SF}$ and $\mathcal{P}$ are axiomatisable, so that $S$ is left
perfect and Condition (A) holds.
 Let $U$ be a right collapsible submonoid of $S$; then
for any $u,v\in U$ we have that $ur=vr$ for some $r\in S$.
We deduce that $U=\{ 1\}$ and $|CU^S|=1$. From
Theorem~\ref{sfomegastab}, the monoid $S$ is an
$\mathcal{SF}$--$\omega$--stabiliser.
\end{proof}

\begin{corollary}\label{X} If $|S|\leq\omega$ and the class $\mathcal P$
is axiomatisable, then $S$ is an $\mathcal P$--$\omega$--stabiliser.
\end{corollary}

\begin{proof} Suppose that $|S|\leq\omega$ and $\mathcal P$ is
axiomatisable. From Theorem~\ref{projax}, $\mathcal {SF}$ is
axiomatisable and $S$ is a left perfect monoid. Hence from
Theorem~\ref{perfect}, $\mathcal{SF=P}$ and $S$ satisfies
Condition $(A)$. Let us construct an embedding $\phi$ of the set
$CU^S$ into the  set $E$ of idempotents of $S$.

Let $U\in CU^S$. From  Lemma~\ref{V},
$U=1/\rho_U$ and from Lemma~\ref{VI}, $\rho_U$ is a
strongly flat left congruence. In view of the equality $\mathcal{SF}=
\mathcal{P}$
there
exist an idempotent $e\in E$ and an $S$-isomorphism
$\alpha:S/\rho_U\rightarrow Se$. Put  $a=1/\rho_U$, so that
$S/\rho_U=Sa$ and $u\,\rho_U\, v$ if and only if $ua=va$.
Now
$\alpha(a)=se$ and $\alpha(ta)=e$ for some $s,t\in S$.
It is easy to see that
consequently, $e=tse$, $eta=ta$ and $seta=a=sta$. Hence
$setset=seet=set$, i.e. $g=set$ is an idempotent of $S$ and
$ga=a$. Moreover, for any $u,v\in S$,
\[\begin{array}{rcl}
u\,\rho_U\, v&\Leftrightarrow& ua=va\\
&\Rightarrow& use=vse\\
&\Rightarrow&ug=vg\\
&\Rightarrow&useta=uga=vga=vseta\\
&\Rightarrow&ua=va\end{array}.\]

Now define $\phi:CU^S\rightarrow E$ by $\phi(U)=g$.
If $\phi(U)=\phi(V)$, then $\rho_U=\rho_V$, so
that $U=1/\rho_U=1/\rho_V=V$. Thus
 $\phi$ is an injection and
 $|CU^S|\leq|E|\leq|S|\leq\omega$.
\end{proof}

\begin{corollary} If $|S|\leq\omega$ and the class
$\mathcal{F}$$r$ is axiomatisable than $S$ is an
$\mathcal{F}$$r$--$\omega$--stabiliser.
\end{corollary}

\begin{proof} If $\mathcal{F}$$r$ is axiomatisable then by
Theorem~\ref{freesax}, $\mathcal P$ is axiomatisable. From
Corollary~\ref{X}, $S$ is a $\mathcal P$--$\omega$--stabiliser, and hence
in
particular an $\mathcal{F}$$r$--$\omega$--stabiliser.
\end{proof}

\end{document}